\documentclass[reqno, 10pt]{amsart}

\usepackage{color}

\usepackage[top=4cm, bottom=2.7cm, left=3cm, right=3cm]{geometry}

\newtheorem{thm}{Theorem}

\newtheorem{lem}[thm]{Lemma}

\newtheorem{ex}[thm]{Example}

\theoremstyle{definition}
\newtheorem{defn}[thm]{Definition}

\theoremstyle{remark}
\newtheorem{rem}[thm]{Remark}

\newcommand{\U}{\mathbf{u}}
\newcommand{\V}{\mathbf{v}}
\newcommand{\w}{\mathbf{w}}

\newcommand{\N}{\mathcal{N}}

\newcommand{\E}{\mathbb{E}}
\newcommand{\p}{\mathbb{P}}

\newcommand{\im}{\mathrm{Im}\:}

\newcommand{\tr}{\mathrm{Tr}\:}

\numberwithin{equation}{section}
\numberwithin{thm}{section}

\newcommand{\be}{\begin{equation}}
\newcommand{\ee}{\end{equation}}
\newcommand{\al}{\alpha}
\newcommand{\e}{\epsilon}
\newcommand{\wt}{\widetilde}

\begin{document}

\title{A Necessary and Sufficient condition for Edge Universality of Wigner matrices}

\author{Ji Oon Lee}
\address{Department of Mathematical Sciences\\
Korea Advanced Institute of Science and Technology\\
Daejeon, 305701, Republic of Korea}
\email{jioon.lee@kaist.edu}
\thanks{Ji Oon Lee is Partially supported by Basic Science Research Program through the National Research Foundation of Korea Grant  2011-0013474}

\author{Jun Yin}
\address{Department of Mathematics\\
University of Wisconsin-Madison\\
480 Lincoln Dr., Madison\\
WI 53706, USA}
\email{jyin@math.wisc.edu}
\thanks{Jun Yin  is Partially supported by  National Research Foundation of U.S.  Grant  DMS-1001655}


\keywords{Edge universality, Tracy-Widon distribution}

\begin{abstract}

In this paper, we prove a necessary and sufficient condition for Tracy-Widom law of Wigner matrices. Consider $N \times N$ symmetric Wigner matrices $H$ with $H_{ij} = N^{-1/2} x_{ij}$, whose upper right entries $x_{ij}$ $(1\le i< j\le N)$ are $i.i.d.$ random variables with distribution $\mu$ and diagonal entries $x_{ii}$ $(1\le i\le N)$ are $i.i.d.$ random variables with distribution $\wt \mu$. The means of $\mu$ and $\wt \mu$ are zero, the variance of $\mu$ is 1, and the variance of $\wt \mu $ is finite. We prove that Tracy-Widom law holds if and only if $\lim_{s\to \infty}s^4\p(|x_{12}| \ge s)=0$. The same criterion holds for Hermitian Wigner matrices. 
\end{abstract}

 
\maketitle{} \thispagestyle{headings}

\section{Background and Main result}

Since the groundbreaking work by Wigner \cite{W}, it has been conjectured and widely believed that local statics of eigenvalues of random matrices are universal in the sense that it depends only on the symmetric class of the ensembles. The universality is one of the most important concepts in random matrix theory, and it can roughly be divided into two different types, the bulk universality and the edge universality. 

Before considering the edge universality, which we will study in this paper, we roughly introduce some important results on bulk universality. The bulk universality concerns the local statistics of eigenvalues in the interior of the spectrum. In the early works of Wigner, Dyson, Gaudin, and Mehta \cite{M, MG, D1, D2}, it was proved that, after proper rescaling, the joint probability density of eigenvalues of Gaussian Unitary Ensemble(GUE) can be explicitly described by the sine kernel, and they conjectured the universallity holds for more general classes of ensembles. For a very general class of invariant ensembles, the bulk universality was proved by Deift et. al. \cite{DKMVZ1, DKMVZ2}, Bleher and Its \cite{BI}, and Pastur and Shcherbina \cite{PS}. Later by Johansson \cite{J2}, the bulk universality was proved for Gaussian divisible ensembles. (See also the work by Ben Arous and P\'ech\'e \cite{BP}.) For general Wigner matrices, a new approach was introduced to prove the bulk universality in a series of papers by   Erd\H{o}s, Schlein, Yau, and others in \cite{ESY1, ESY2, ESY3, ESY4, ESYY, EYY1, EYY2, EYY, EKYY1, EKYY2}. The bulk university for Wigner matrices was also obtained by Tao and Vu \cite{TV1}. See the reviews \cite{EY, EY2} for further discussion.

The distribution of the largest eigenvalue exhibits another type of universality, which is called the edge universality. Let $\lambda_N$ be the largest eigenvalue of a Wigner matrix. For the Gaussian ensembles, the distribution function of $\lambda_N$ was first identified by Tracy and Widom \cite{TR1, TR2}. More precisely, it is proved that
\begin{equation}
\lim_{N \to \infty} \p (N^{2/3} (\lambda_N - 2) \leq s ) = F_{\beta} (s),
\end{equation}
where the Tracy-Widom distribution functions $F_{\beta}$ can be discribed by Painlev\'e equations, and $\beta = 1, 2, 4$ corresponds to Orthogonal/Unitary/Symplectic ensemble, respectively. The joint distribution of $k$ largest eigenvalues can be expressed in terms of the Airy kernel, which was shown by Forrester \cite{F}. If we denote the $k$ largest eigenvalues by $\lambda_N, \lambda_{N-1}, \cdots, \lambda_{N-k+1}$, then for Gaussian ensembles, the joint distribution function of rescaled eigenvalues has the limit
\begin{equation} \begin{split} \label{eq:TW}
& \lim_{N \to \infty} \p \left( N^{2/3} (\lambda_{N} -2) \leq s_1, N^{2/3} (\lambda_{N-1} -2) \leq s_2, \cdots, N^{2/3} (\lambda_{N-k+1} -2) \leq s_k \right) \\
 = &\;F_{\beta, k} (s_1, s_2, \cdots, s_k),
\end{split} \end{equation}
which will also be called the Tracy-Widom distribution.

The condition for \eqref{eq:TW} has been studied intensively. In the direction of sufficient condition, it has been improved as follows: The result \eqref{eq:TW} was first extended to general Wigner matrices by Soshnikov \cite{S1} with the condition that all odd moments of matrix entries vanish (e.g. the symmetric distribution) and with the Gaussian decay. Ruzmaikina \cite{R} showed that the Gaussian decay can be replaced with polynomial decay faster then $x^{-18}$. Also, under the condition that matrix entries are symmetrically distributed, Khorunzhiy \cite{OK} proved a bound for the spectral norm for the matrices whose entries have finite $12+o(1)$ moment. For the non-symmetric case, \eqref{eq:TW} is proved in \cite{TV2} by Tao and Vu with the condition that matrix entries have vanishing third moment and sub-expontential decay. (Some partial results in the non-symmetric case can be found in \cite{P2} and \cite{P3}.) Later, the vanishing third moment condition was removed by Erd\H{o}s, Yau, and others in \cite{EYY, EKYY2}, i.e., \eqref{eq:TW} is implied by the sub-expontential decay condition. The current best sufficient condition for \eqref{eq:TW}, as we know, is that the matrix entries have finite $12+o(1)$ moments, which was proved in \cite{EKYY2}. Numerical results by Biroli, Bouchaud, and Potters \cite{BBP} predicted that the Tracy-Widom distribution would appear when the $(4 + \epsilon)$-th moment is finite.

On the other hand, for the Wigner matrices whose entries have heavy tails, the necessary condition for \eqref{eq:TW} is studied as follows: In the case of real symmetric matrices with i.i.d. entries, it was proved by  Soshnikov  \cite{S2} that, when the variance of entries diverges, the largest eigenvalue has Poisson statistics. More precisely, in \cite{S2} was considered the case where the distribution of entries satisfies
\begin{equation} \label{eq:tail}
\p (|h_{ij}| > x) = \frac{h(x)}{x^{\alpha}},
\end{equation}
where $h(x)$ is a slowly varying function and $\alpha < 2$. The case $2 \leq \alpha < 4$ was later studied by Auffinger, Ben Arous, and P\'ech\'e \cite{ABP}, which also shows the Poisson statistics. We also remark that in the case $\alpha < 2$ the Wigner semi-circle law no longer holds in the bulk. See the work by Ben Arous and Guionnet \cite{BG} for more detail. The numerical simulation results in \cite{BBP} also suggest that $\alpha = 4$ in \eqref{eq:tail} will provide the marginal case. 

The edge universality has been generalized in many directions, for example, for the sample covariance matrices \cite{J1, Jo, So1, S3, PY1} and for correlation matrix \cite{BPZ, PY2}. For the deformed matrices, which are described as a finite rank perturbation of sample covariance matrices and the deformed Wigner matrices, the Tracy-Widom law also holds when the outliers are excluded \cite{BBaP, BV1, BV2, P, KY}.


In this paper, we prove the following simple criterion on this property: The necessary and sufficient condition for the joint probability density of the $k$ largest eigenvalues of a Wigner matrix (see definition in Def. \ref{def: SWigner}) to weakly converge to that of Gaussian ensembles, i.e., the Tracy-Widom distribution, is that the off-diagonal entry of the Wigner matrix satisfies
\be \label{crit}
\lim_{s\to +\infty} s^4 \p(|x_{12}|\ge s)=0.
\ee
We note that this criterion has not been predicted in any previous works.

The precise definition of the Wigner matrix we consider in this paper is as follows:
\begin{defn} \label{def: SWigner}
The (standard) symmetric (Hermitian) Wigner matrix of size $N$ is a symmetric (Hermitian) matrix 
$$
(H_N)_{ij}=h_{ij}=\frac{1}{\sqrt N} x_{ij},\quad  1\le i, j\leq N,
$$
where the upper-triangle entries $(x_{ij})$ $(1\le i\le j\leq N )$ are independent real (complex) random variables with mean zero satisfying the following conditions:
\begin{itemize}
\item The upper right entries $x_{ij}$ $(1\le i< j\le N)$ are $i.i.d.$ random variables with distribution $\mu$, satisfying $\E x_{12}=0$ 
and $\E|x_{12}|^2=1$. 
\item The diagonal entries $x_{ii}$ $(1\le i\le N) $ are $i.i.d.$ random variables with distribution $\wt \mu$, satisfying $\E x_{11}=0$ 
and $\E|x_{11}|^2 < \infty$.
\item In addition, for the Hermitian case, $\E (x_{12})^2=0$.
\end{itemize}

When the random variables $x_{ij}$ and $x_{ii}$ are real Gaussian with $\E|x_{ii}|^2 = 2$, $H$ will be called Gaussian Orthogonal Ensemble (GOE). Similarly, when $x_{ij}$ are complex Gaussian and $x_{ii}$ are real Gaussian with $\E|x_{ii}|^2 = 1$, $H$ will be called Gaussian Unitary Ensemble (GUE). We denote by $\lambda_1 \leq \lambda_2 \cdots \leq \lambda_N$ the eigenvalues of $H_N$ and by ${\bf u}_1, {\bf u}_2 \cdots {\bf u}_N$ the corresponding eigenvectors of $H_N$. 
\end{defn}

The main result of this paper is the following theorem:
\begin{thm} \label{main result}
For any centered distribution $\nu$ and $\wt{\nu}$ with variance $1$ and finite variance, respectively, let $H_N$ be the Wigner matrix defined in Definition \ref{def: SWigner} such that $x_{12}$ and $x_{11}$ have distributions $\nu$ and $\wt{\nu}$, respectively. Then,
\begin{itemize}
\item Sufficient condition: if \eqref{crit} holds, then for any fixed $k$, the joint distribution function of $k$ rescaled largest eigenvalues,
\be\label{def: ranvec}
\p \left( N^{2/3} (\lambda_{N} -2) \leq s_1, N^{2/3} (\lambda_{N-1} -2) \leq s_2, \cdots, N^{2/3} (\lambda_{N-k+1} -2) \leq s_k \right)
\ee
has a limit as $N \to \infty$, which coincides with that in the GUE (GOE) case, i.e., it weakly converges to the Tracy-Widom distribution. 

\item Necessary condition: if \eqref{crit} does not hold, then the joint distribution function \eqref{def: ranvec} does not converge to the Tracy-Widom distribution. Furthermore, we have 
\be \label{nec}
\limsup_{N\to \infty} \, \p(\lambda_N( H_N ) \ge 3) > 0.
\ee
\end{itemize}
\end{thm}

\begin{rem} \label{example}
While any distribution with finite fourth moment satisfies the criterion \eqref{crit}, the converse is not true. If we consider, for example, the distribution whose density $f(x)$ decays as $|x|^{-5} \log |x|$, then it does not have finite fourth moment though \eqref{crit} holds for it. The existence of this particular example, however, does not contradict the result in \cite{BY}, which proved that $\lim_{N \to \infty} \lambda_N(H_N) =2$ a.s. if and only if the fourth moment exists.
\end{rem}

Our result provides a very simple sufficient and necessary condition for the edge universality of Wigner matrices without assuming any other  properties of matrix entries. This also shows the existence of four moments, which was predicted to be needed for the edge universality, is not necessary for the Tracy-Widom result, as we can see from Remark \ref{example}.

Our proof of the main result features two key observations.

1. If we introduce a `cutoff' on each matrix element at $N^{-\epsilon}$, then the matrix with the cutoff can well approximate the original matrix in terms of the behavior of the largest eigenvalue if and only if the criterion \eqref{crit} holds.

2. The Green function comparison method (e.g. Theorem 6.3 in \cite{EYY}), which was first introduced in \cite{EYY1}, can be extended to the random matrices whose entries have a bounded support of size $N^{-\epsilon}$ for some $\epsilon > 0$. The Green function comparison method was applied on studying the distribution of the eigenvalues of the Wigner matrices, deformed Wigner matrices, covariance matrices, correlation matrices, and adjacency matrices of random graphs \cite{EYY1, EYY2, EYY, KY, PY1, PY2, EKYY1, EKYY2}. It was also used in the study of the distribution of eigenvectors \cite {KY0} and the determinant \cite{TV3} of Wigner matrices. We believe that our new method in the present paper can be used to improve the results in these topics.

The first observation can be understood in the framework of the deformed Wigner matrix. We consider the matrix with the cutoff as the unperturbed part and the remaining part the perturbation. As studied in \cite{BBaP, P, BV1, BV2, KY}, if the perturbation is small enough, then we can predict the behaviors of the largest eigenvalues of the original matrix from the matrix with cutoff. On the other hand, if the original/perturbation matrix has an entry whose absolute value is larger than $1$, then the matrix will have an eigenvalue greater than $2$, hence the Tracy-Widom distribution fails. Roughly speaking, the criterion \eqref{crit} means that each off-diagonal entry is bounded by $1$ with probability $1 - o(N^{-2})$, thus the condition gaurantees that no entries are larger than $1$ with probability $1 - o(1)$. We remark that a similar argument was introduced in \cite{BBP}.

The Green function comparison part is more technical. Given the matrix with the cutoff at $N^{-\epsilon}$, we first find a `better' matrix, in the sense that it is already known to satisfy Tracy-Widom law, whose first four moments coincide those of the given matrix. We then apply Lindeberg replacement stretagy sufficiently many times (more precisely, $O(\e^{-1})$ times) to compare the Green functions. The basic idea is as follows: Using Green function comparison method, one can study the difference of the functional on Green functions between the `better' matrix and the original matrix. Instead of bounding the difference directly, however, we represent it as a new functional, which is much more complicated, on Green functions, with gaining a factor $N^{-\e}$. This new functional can be easily bounded for the `better' matrix case, but not for original matrix. To solve this issue, again we use Green function comparison method to estimate the difference of this new functional between the `better' matrix and original matrix. Repeating this process, we obtain the desired bound. The details will be explained later.

Though the Green function comparison method has been used in previous papers, it was always required to have a good bound on Green function with high probability. This is one of the reasons that the distribution of matrix entries have been assumed to satisfy subexponential decay condition  in many papers. In this paper, however, we show a way to circumvent this problem, which can be used to achieve many other results, besides the edge universality, for heavy-tailed random matrices. See, for example, the rigidity result in Theorem \ref{thm: new rigidity} that holds for the random matrices whose entries are only bounded by $N^{-\epsilon}$ for some $\epsilon > 0$. {(Note: it is also an interesting result, since it shows that the locations and the fluctuations of the eigenvalues keep unchanged, even if the fluctuations of the matrix entries become very large, i.e., from $N^{-1/2}$ to $N^{-\e}$. )}

This paper is organized as follows. In Sections \ref{sec2} and \ref{sec3}, we introduce the notations and collect tools we use to prove the main result. In Section \ref{sec4}, we prove the main result using the cutoff argument. Technical results on the Green function comparison method will be proved in Sections \ref{sec6} and \ref{sec7}.

\begin{rem}
In this paper, for simplicity, we will prove Theorem \ref{main result} only for the real symmetric case with $k=1$. The general case can be proved analogously. 
\end{rem}

\section{Notations} \label{sec2}
 
In the proof, we will use some variations of standard Wigner matrix defined in Definition \ref{def: SWigner}, which are defined as follows:

\begin{defn}[Generalized symmetric Wigner matrix]
A symmetric matrix $H_N$ is said to be a generalized symmetric Wigner matrix of size $N$ if its upper-triangular entries
$$
(H_N)_{ij} = h_{ij}, \quad  1 \leq i \leq j\leq N,
$$ 
are independent real random variables with mean zero, {\it whose distribution may depend on $i,j$ and $N$}, and satisfy, for some constant $C_0$,
\begin{equation} \label{eq:diag}
|\E h_{ij}^2 - N^{-1} |\le C_0 N^{-1} \delta_{ij}.
\end{equation}

\end{defn}

\begin{rem}
The results on generalized Wigner matrices, especially the constants in the results, may depend on $C_0$, but we will not emphasize it in the sequel.
\end{rem}

As in \cite{EKYY1, EKYY2}, we will use the following definition to characterize events of very high probability.

\begin{defn}[High probability events]\label{def: hprob}
Define
\be
\varphi := (\log N)^{\log \log N}.
\ee
We say that an $N$-dependent event $\Omega$ holds with $\zeta$-high probability if there exist constants $c, C>0$, independent of $N$, such that
\begin{equation}
\p (\Omega ) \geq 1- N^C e^{-c \, \varphi^\zeta }
\end{equation}
for all sufficiently large $N$. For simplicity, for the case $\zeta=1$, we just say high probability. 
\end{defn}

The next condition on the distributions of the matrix entries will be used in the proof.
 
 \begin{defn}[Bounded support condition] \label{cond:bdd}
We say a family of random matrices $(H_N)_{ij} = (h_{ij})$ satisfies the bounded support condition with $q$, if for $1 \leq i, j \leq N$
\begin{equation} \label{eq:cond_bdd}
|h_{ij}| \leq q^{-1}
\end{equation}
with probability larger than $1 - e^{-N^c}$ for some $c > 0$. Here, $q$ may depend on $N$ and usually $N^{\phi} \leq q \leq N^{1/2} (\log N)^{-1}$ for some $\phi > 0$.
\end{defn}

Note that the Gaussian distribution satisfies bounded support condition with any $q < N^{\phi}$ for any $\phi < 1/2$. We also remark that, when $H_N$ satisfies the bounded support condition, the event $\{ |h_{ij}| \leq q^{-1} \}$ holds with `very' high probability, i.e., it holds with $\zeta$-high probability for any positive constant $\zeta$. For this reason, the extreme event $\{ |h_{ij}| \geq q^{-1} \}$ is negligible, and throughout the paper, we will not consider the case it happens.


  
\begin{defn}[Green function, semicircle, $m_{sc}$ and $m$]
For a Wigner matrix $H$, we define the Green function of $H$ by 
\be \label{green}
G_{ij}(z) := \left(\frac1{H-z}\right)_{ij}, \quad z=E+i\eta, \quad E \in \mathbb R, \quad \eta>0.
\ee
The Stieltjes transform of the empirical eigenvalue distribution of $H$ is given by  
\be \label{mNdef}
m(z) = m_N (z): = \frac{1}{N} \sum_j G_{jj}(z) = \frac{1}{N} \tr\, \frac{1}{H-z}\,.
\ee
Define $m_{sc} (z)$ as the unique solution of
\be\label{defmsc} 
m_{sc} (z) + \frac{1}{z+m_{sc} (z)} = 0,
\ee
with positive imaginary part for all $z$ with $\im z > 0$, i.e.,
\be\label{temp2.8}
m_{sc}(z)=\frac{-z+\sqrt{z^2-4}}{2},
\ee
where the square root function is chosen with a branch cut in the segment $[-2,2]$ so that asymptotically $\sqrt{z^2-4}\sim z$ at infinity.
This guarantees that the imaginary part of $m_{sc}$ is non-negative for $\eta=\im z > 0$ and in the $\eta\to 0$ limit it is the Wigner semicircle distribution
\be \label{def:sc}
\varrho_{sc}(E) : = \lim_{\eta\to 0+0}\frac{1}{\pi}\im \, m_{sc}(E+i\eta) = \frac{1}{2\pi} \sqrt{ (4-E^2)_+}.
\ee
We will also frequently use the notations
$$
z : = E + i \eta, \quad \kappa: = \big| \, |E|-2 \big|.
$$
\end{defn}

The following lemma (Lemma 4.2 of \cite{EYY2}) collects elementary properties of the Stieljes transform of the semicircle law. As a technical note, we use the notation $f \sim g$ for two positive functions in some domain $D$ if there exists a positive universal constant $C$ such that $C^{-1}\le f(z)/g(z) \le C$ holds for all $z\in D$.

\begin{lem} \label{lm:msc}  
We have for all $z$ with $\im z>0$ that
\be \label{zmsc2}
|m_{sc}(z)| = |m_{sc}(z)+z|^{-1} \le 1.
\ee
Let $z=E+i\eta$ with $|E|\le 5$ and $0< \eta \le 10$. We have
\be \label{smallz}
|m_{sc}(z)|\sim 1, \qquad   |1-m_{sc}^2(z)|\sim \sqrt{\kappa+\eta}
\ee
and the following two bounds: 
\begin{align} \label{esmallfake}
\im m_{sc} (z) \sim
	\begin{cases}
	\displaystyle \frac{\eta}{\sqrt{\kappa+\eta}} & \mbox{if  $\kappa\ge\eta$ and $|E|\ge 2$}, \\
	 & \\
	\sqrt{\kappa+\eta} & \mbox{if $\kappa\le \eta$ or $|E|\le 2$.}
	\end{cases}
\end{align}
\end{lem}

\begin{defn}[Classical location of the eigenvalue]
We denote by $\gamma_j$ the {\it classical location} of the $j$-th eigenvalue, i.e., $\gamma_j$ is defined by 
\be \label{def:gamma}
N \int_{-\infty}^{\gamma_j} \varrho_{sc}(x) {\rm d} x = j, \qquad 1\leq j\le N.
\ee
\end{defn}

\begin{rem}
Throughout the paper, the notations $O( \cdot )$, $o( \cdot )$, and $ \ll$ will always be with respect to the limit $N \to \infty$, where $a \ll b$ means $a = o(b)$. The constant $C$ will denote various constants independent of $N$.
\end{rem}

\section{Tools} \label{sec3}
In this section, we introduce some results that will be used in the proof of the main theorem. Some of them are already proved in previous papers with H.-T. Yau, L. Erd\H{o}s, and A. Knowles, and we made slight changes in the statement to fit the notations and definitions in this paper. We also extend some of the known results.
 
Define the domain
\be \label{defS}
{\bf  S}(C) = \Big\{ z=E+i\eta\; : \; |E|\leq 5, \quad N^{-1}\varphi^C < \eta \le 10 \Big\}.
\ee

\begin{lem}[Previous results on generalized Wigner matrix] \label{lem: old}
Let $H$ be a generalized Wigner matrix satisfying  bounded support condition with $q$. There exists a constant $C>0$ such that, if $q\ge\varphi^C$, 
then the following properties hold with 3-high probability:

\begin{enumerate}
\item {\bf Local semicircle law} (Theorem 2.8 in \cite{EKYY1}) : 
\begin{equation} \label{scm}
\bigcup_{z\in {\bf  S }(C) } \Big\{ |m(z)-m_{sc}(z)| 
\leq \varphi^C 
\left( 
\min \left\{ 
\frac{1}{\sqrt{\kappa +\eta}}\frac{1}{q^2}, \frac{1}{q}
\right\} + \frac{1}{N \eta}
\right)
\Big\}
\end{equation} 
 \begin{align}\label{Gij estimate}
\bigcup_{z\in {\bf  S }(C)  }
 \left\{
 \max_{ij} |G_{ij}(z)-\delta_{ij}m_{sc}(z)| \le
\varphi^C
\left(
\frac{1}{q} + \sqrt{\frac{\im m_{sc}(z)}{N \eta}} + \frac{1}{N \eta} 
\right)
\right\}
\end{align}

\item {\bf Bound on $\|H\|$} (Lemma 4.4 in \cite{EKYY1}) :

\be\label{sjsy}
\|H\|\leq 2+\varphi^C(q^{-2}+N^{-2/3})
\ee

\item {\bf Delocalization} (Remark 2.18 in \cite{EKYY1}) :
 
\be\label{deloc}
\max_{\alpha, i} \left|{\bf u}_\al(i)\right|^2\leq \frac{\varphi^C}{N}
\ee

\end{enumerate}

Furthermore, if  $q\ge N^{\phi}$ for some constant $\phi > 1/3$, then the following properties hold with 3-high probability:

\begin{enumerate} \setcounter{enumi}{3}
\item {\bf Rigidity of the eigenvalues}  (Theorem 2.13 and Remarks 2.14, 2.15 in \cite{EKYY1}) :
\be
\bigcup_j\bigg\{|\lambda_j-\gamma_j| 
{ \,\leq\,} \varphi^C \left( \Big [ \min \big ( \, j ,  N-j+1 \,  \big) \Big  ]^{-1/3} N^{-2/3} +q^{-2}\right)\bigg\} .
\ee

\item {\bf Bound on $G_{ij}$ out of Spectrum} (Equations (3.32), (3.58) and (4.36)-(4.46) in \cite{EKYY1}) : For any large $a>12 $, there exists a constant $C>0$ depending on $a$ such that for $z=E+i\eta$ with
\be\label{3737}
2+\varphi^CN^{-2/3}\leq E\leq 3, \quad \eta=\varphi^{(3+a)} N^{-1}\kappa^{-1/2},
\ee
we have 
\be\label{3838}
|m(z)-m_{sc}(z)|\leq \frac{1}{\varphi^aN\eta }, \quad \im m(z)\leq \frac{1}{\varphi^aN\eta },  \quad \max_{i\ne j}|G_{ij}|\leq \frac{1}{\varphi^{a/3} N \eta}
\ee

\end{enumerate}

\end{lem}

\begin{proof}[Proof of Lemma \ref{lem: old}] For the case $C_0=0$ (see \eqref{eq:diag}), these results except \eqref{3838} were already proved with the choice of $\xi= 3\log\log N$ in \cite{EKYY1}. Furthermore, the proofs in \cite{EKYY1} can be extended to the case $C_0=O(N^{-1})$ with almost no revision. Heuristically speaking, it only brings the error of order $N^{-1}$. As we can see from the proofs, these inequalities still hold after multiplying $G$, $\lambda$, and $\bf u$   by a factor of $1+O(N^{-1})$.

In order to prove \eqref{3838}, we choose $\xi= 3\log\log N $ as above and let $C_1=1$ in (2.15) of \cite{EKYY1} so that $(\log N)^{C_1\xi}= \varphi^3$. In (4.36)-(4.46) of \cite{EKYY1}, it was actually proved that (see (4.45) and (4.46) of \cite{EKYY1}), with 3-high probability,
$$
|m(z)-m_{sc}(z)|\leq \frac{1}{(\log N)N\eta },\quad \im m(z)\leq \frac{1}{(\log N)N\eta }.
$$
where $\Lambda:=|m-m_{sc}|$ and $\al\sim \sqrt{\kappa}$ in \cite{EKYY1}. Hence, we only need to change the exponent to obtain the first two parts of \eqref{3838}. To achieve that, one can replace $\varphi^3 (\log N)^2$ in (4.37) of \cite{EKYY1} with $\varphi^{3+2a}$, and replace $\varphi^3 (\log N)$ in (4.38), (4.42), and the inequality below (4.44) of \cite{EKYY1} with $\varphi^{3+ a}$. Then, as in (4.40), (4.45), and (4.46) of \cite{EKYY1}, we obtain the first two terms of \eqref{3838}

Now we prove the third part of \eqref{3838}. Using (3.32) and (3.58) of \cite{EKYY1}, we have that with 3-high probability, 
\be\label{tyzz}
\max_{i\ne j} |G_{ij}| \leq O(q^{-1})+\frac{1}{\varphi^{(a/2)-2} N \eta}
\ee
From \eqref{3737}, we have $\eta \ge \varphi^{O(1)}N^{-2/3}$. Together with assumptions on $q$ and \eqref{tyzz}, we obtain the third part of \eqref{3838} and complete the proof.

\end{proof}

\begin{rem} \label{prozeta}
In \cite{EKYY1}, the author proved  that, for any fixed $\zeta$, there exists $C_\zeta$ such that when $C = C_{\zeta}$ the above statements hold with $\zeta$-high probability. In this paper, we only use $1 \le \zeta \le 3$ in the   proofs. We note that, however, similar results to Lemma \ref{lem: old} holds with $\zeta$-high probability for some constant $C=C_\zeta$.
\end{rem}

\begin{thm}[Edge universality on generalized Wigner matrix: Theorem 2.7 in \cite{EKYY2}] \label{thm: old edge un}
Let $H^{\w}$ be a GOE and $H^{\V}$ a generalized symmetric Wigner matrix with 
$$
\E^{\V} |h_{ij}|^2 = N^{-1} + \delta_{ij} N^{-1}.
$$
Assume that $H^{\V}$ satisfies the bounded support condition with $q=N^{\phi}$, for some constant $1/3 < \phi \leq 1/2$. Then, there exists a constant $\delta > 0$ such that, for any $s \in \mathbb{R}$, we have
\begin{equation} \label{qset}
\p^{\w} \Big( N^{2/3} (\lambda_N -2) \leq s - N^{-\delta} \Big) - N^{-\delta} \leq \p^{\V} \big( N^{2/3} (\lambda_N -2) \leq s \big) \leq \p^{\w} \Big( N^{2/3} (\lambda_N -2) \leq s + N^{-\delta} \Big) + N^{-\delta}.
\end{equation}
Here, $\p^{\V}$ and $\p^{\w}$ denote the laws of the ensembles $H^{\V}$ and $H^{\w}$, respectively.
\end{thm}
  
\begin{rem} \label{genedge} As in \cite{EYY} and \cite{EKYY2}, Theorem \ref{thm: old edge un}, as well as Lemma \ref{lem: ex edge un} and Theorem \ref{thm: new edge un} below, can be extended to finite correlation functions of extreme eigenvalues. For example, we have the following extension:
 \begin{align}\label{twa}
&  \p^{\bf w} \Big ( N^{2/3}  ( \lambda_N -2) \le s_1- N^{-\e}, \ldots, N^{2/3} ( \lambda_{N-k} -2) 
\le s_{k+1}- N^{-\e} \Big )- N^{-\delta}   \nonumber \\
& 
 \le   \p^{\bf v} \Big ( N^{2/3} (  \lambda_N -2) \le s_1,  \ldots, N^{2/3} ( \lambda_{N-k} -2) 
\le s_{k+1}  \Big )  \\
 &  \le 
 \p^{\bf w} \Big ( N^{2/3} ( \lambda_N -2) \le s_1+ N^{-\e}, \ldots, N^{2/3} ( \lambda_{N-k} -2)
 \le s_{k+1}+ N^{-\e}   \Big )+ N^{-\delta}  \nonumber 
\end{align}  
for all $k$ fixed and $N$ sufficiently large. The proof of \eqref{twa} is similar to that of \eqref{qset} except that it uses the general form of the Green function comparison theorem. 
\end{rem}

We slightly extend this result as follows:

\begin{lem}\label{lem: ex edge un}
Assume the same condition as in Theorem \ref{thm: old edge un} except that we assume $\E^\V |h_{ii}|^2 \leq  C_0/N$ for some uniform $C_0$ instead. Then, the conclusion of Theorem \ref{thm: old edge un} still holds.
\end{lem}

We postpone the proof of this lemma to the end of this section.

\bigskip

To prove our main result, we claim the following three lemmas, which extend the previous results to the Wigner matrix with bounded support condition of small $q$. First, we improve the previous result on rigidity. We define the {\it normalized empirical counting function} by
$$
{\mathfrak n}(E):= \frac{1}{N}\# \{ \lambda_j\le E\}.
$$
Let
$$
n_{sc}(E) :  = \int_{-\infty}^E \varrho_{sc}(x) {\rm d} x
$$
be the distribution function of the semicircle law.

\begin{thm}[Rigidity of eigenvalues: small $q$ case] \label{thm: new rigidity}
Let $H$ be a generalized symmetric Wigner matrix with some constant $C_1$ such that for any $1 \le i< j\leq N$
$$
\E |h_{ii}|^2 \leq C_1/N, \quad \E  |h_{ij}|^2=1/N, \quad \E |h_{ij}|^3 \leq C_1 N^{-3/2}, \quad \E  |h_{ij}|^4 \leq C_1(\log N) N^{-2}
$$ 
and $H$ satisfies the bounded support condition with $q=N^{\phi}$ for some constant $\phi>0$. Then, there exist constants $C>0$ and $N_0$, depending only on $C_1$ and $\phi$, such that with high probability we have
\be \label{rigiditynew}
\bigcup_j \bigg\{ |\lambda_j-\gamma_j| {\,\leq\,} \varphi^C \Big [ \min \big ( \, j ,  N-j+1 \, \big) \Big]^{-1/3} N^{-2/3} \bigg\} 
\ee
and
\be\label{rigiditynew2}
\sup_{|E|\le 5} \big| {\mathfrak n} (E)-n_{sc}(E)\big|  {\,\leq\,}  \frac{\varphi^C}{N} 
\ee
for any $N > N_0$.
\end{thm}

Next theorem shows that the edge universality holds under the assumptions in Theorem \ref{thm: new rigidity}.

\begin{thm}[Edge universality: small $q$ case] \label{thm: new edge un}
Let $H^{\w}$ be a GOE and $H^\V$ be a generalized symmetric Wigner matrix satisfying the conditions for $H$ in Theorem \ref{thm: new rigidity}. Then, there exists a constant $\delta>0$ such that for any $s\in \mathbb R$, we have
\begin{equation} \label{qset1}
\p^{\w} \Big( N^{2/3} (\lambda_N -2) \leq s - N^{-\delta} \Big) - N^{-\delta} \leq \p^{\V} \big( N^{2/3} (\lambda_N -2) \leq s \big) \leq \p^{\w} \Big( N^{2/3} (\lambda_N -2) \leq s + N^{-\delta} \Big) + N^{-\delta}.
\end{equation}
Here, $\p^{\V}$ and $\p^{\w}$ denote the laws of the ensembles $H^{\V}$ and $H^{\w}$, respectively.
\end{thm}

Finally, we show a weak bound on $G_{ij}$ ($i\neq j$) of $H$ satisfying the conditions in Theorem \ref{thm: new rigidity}. 
\begin{lem}[Bounds on $G_{ij}$: small $q$ case] \label{lem: WBGij}
Let $H $ be a generalized symmetric Wigner matrix satisfying the conditions for $H$ in Theorem \ref{thm: new rigidity}. Then, for any $0 < c < 1$, $z=E+i\eta$ with $|E|\le 5$, and $\eta\ge N^{-1+c}$, we have the following weak bound on $G_{ij}$ $(i\ne j)$:
\be\label{res: WBGij}
\E|G_{ij}(z)|^2 \leq \varphi^C \left(  \frac{\im m_{sc}(z)}{N \eta}  + \frac{1}{(N \eta)^2}\right).
\ee
\end{lem}

In the remainder of this section, we give the proof of Lemma \ref{lem: ex edge un}.

\begin{proof}[Proof of Lemma \ref{lem: ex edge un}]
This lemma is a simple extension of Theorem 2.7 of \cite{EKYY2}. Thus, from the proof of Theorem 2.7 in \cite{EKYY2}, we find that it suffices to prove the following claim, which corresponds to Proposition 6.6 of \cite{EKYY2}:

\textit{Claim.} Let $F: \mathbb{R} \to \mathbb{R}$ be a function whose derivatives satisfy
\begin{equation}\label{jhw}
\sup_x | F^{(n)} (x) | (1+ |x|)^{-C_1} \leq C_1, \quad n = 1, 2, 3, 4,
\end{equation}
with some constant $C_1 > 0$. Then, there exists a constant $\widetilde{\epsilon} > 0$, depending only on $C_1$, such that, for any $\epsilon < \widetilde{\epsilon}$ and for any real numbers
\begin{equation}
E, E_1, E_2 \in \{ x : |x-2| \leq N^{-2/3 + \epsilon} \} =: I_{\epsilon},
\end{equation}
and setting $\eta := N^{-2/3 - \epsilon}$, we have
\begin{equation}\label{sbc}
\Big| \E^{\V} F \left( N \eta \: \im m (E + i \eta) \right) - \E^{\w} F \left( N \eta \: \im m (E + i \eta) \right) \Big| \leq N^{1/3 + C \epsilon} q^{-1}
\end{equation}
and
\begin{equation}\label{sbc2}
\left| \E^{\V} F \left( N \int_{E_1}^{E_2} dy \: \im m (y + i \eta) \right) - \E^{\w} F \left( N \int_{E_1}^{E_2} dy \: \im m (y + i \eta) \right) \right| \leq N^{1/3 + C \epsilon} q^{-1}
\end{equation}
for some $C$ and for any sufficiently large $N$.

We only prove \eqref{sbc}, and \eqref{sbc2} can be proved similarly. In order to prove the claim, we only need to prove
\begin{equation} \label{eq:claim_old}
\left| \E F \left( \eta^2 \sum_{i \neq j} G^{\V}_{ij} \overline{G^{\V}_{ij}} \right) - \E F \left( \eta^2 \sum_{i \neq j} G^{\w}_{ij} \overline{G^{\w}_{ij}} \right) \right| \leq C N^{1/3 + C \epsilon} q^{-1}.
\end{equation}
(See the proof of Theorem 6.3 of \cite{EYY} for more detail.) 

Fix a bijective ordering map on the index set of the independent matrix elements,
\begin{equation} \label{eq:ordering}
\Phi: \{ (k, \ell) : 1 \leq k \leq \ell \leq N \} \to \{ 1, \cdots, \gamma_{\max} = \frac{N(N+1)}{2} \},
\end{equation}
and let $H_{\gamma}$ be the Wigner matrix whose matrix elements $h_{k \ell}$ follows the distribution $h_{k \ell}^{\V}$ if $\Phi(k, \ell) \leq \gamma$ and the distribution $h_{k \ell}^{\w}$ otherwise. In particular, $H_0 = H^{\w}$ and $H_{\gamma_{\max}} = H^{\V}$. (Note that the index used here is slightly different from previous papers.) Since the Gaussian distribution satisfies the bounded support condition with $q$, we remark that $H_{\gamma}$ satisfies bounded support condition with $q$ for any $\gamma = 0, 1, \cdots, \gamma_{\max}$.

For simplicity, let
\begin{equation}
G^{\gamma}_{ij} := \left( \frac{1}{H_{\gamma} -z} \right)_{ij}.
\end{equation}
Note that matrices $H_{\gamma}$ and $H_{\gamma -1}$ differ only at $(a, b)$ and $(b, a)$ elements, where $\Phi (a, b) = \gamma$. Let $v_{ab} := h_{ab}^{\V}$ and $w_{ab} := h_{ab}^{\w}$. We define matrices $V$ and $W$ by
\begin{equation}
V_{k \ell} := \delta_{ak} \delta_{b \ell} v_{ab} + \delta_{a \ell} \delta_{bk} v_{ba}, \quad W_{k \ell} := \delta_{ak} \delta_{b \ell} w_{ab} + \delta_{a \ell} \delta_{bk} w_{ba},
\end{equation}
so that we can rewrite $H_{\gamma}$ and $H_{\gamma -1}$ as
\begin{equation}
H_{\gamma} = Q + V, \quad H_{\gamma -1} = Q + W,
\end{equation}
with a matrix $Q$ satisfying $Q_{ab} = Q_{ba} = 0$.

Define Green's functions
\begin{equation} \label{eq:green_fns}
R := \frac{1}{Q-z}, \quad S := \frac{1}{H_{\gamma} - z}, \quad T := \frac{1}{H_{\gamma -1} - z}.
\end{equation}
Note that we have a priori estimates
\begin{equation} \label{eq:S_bound}
\max_{k, \ell} \max_{E \in I_{\epsilon}} |S_{k \ell} (E + i \eta) - \delta_{k \ell} m_{sc} (E + i \eta)| \leq C N^{-1/3 + 2\epsilon}
\end{equation}
from part (1) of Lemma \ref{lem: old}, and
\begin{equation} \label{eq:R_bound}
\max_{k, \ell} \max_{E \in I_{\epsilon}} |R_{k \ell} (E + i \eta) - \delta_{k \ell} m_{sc} (E + i \eta)| \leq C N^{-1/3 + 2\epsilon}
\end{equation}
with high probability. To see \eqref{eq:R_bound}, we first expand $R_{k \ell}$ using the resolvent expansion
\begin{equation}\label{8800}
R = S + SVS + (SV)^2 S + (SV)^3 S + (SV)^4 S + (SV)^5 R.
\end{equation}
Since $V$ has at most two non-zero entries, each term in the expansion can be written as a sum of finitely many terms consisting of the entries of $S$, $V$, and $R$. 
From the bound \eqref{eq:S_bound}, the fact that $v_{ab}$ satisfies bounded support condition with $q$, and the trivial bound $R_{ij} \leq \eta^{-1} \leq N$, we obtain the estimate \eqref{eq:R_bound}.

When $a \neq b$, from the proof of Proposition 6.6 in \cite{EKYY2} we have that
\begin{equation} \label{eq:S-T 1}
\left| \E F \left( \eta^2 \sum_{i \neq j} S_{ij} \overline{S_{ij}} \right) - \E F \left( \eta^2 \sum_{i \neq j} T_{ij} \overline{T_{ij}} \right) \right| \leq C N^{-5/3 + C \epsilon} q^{-1}.
\end{equation}

Consider the case $a = b$. Using the resolvent expansion
\begin{equation}
S = R - RVR + (RV)^2 S,
\end{equation}
we find that
\begin{equation} \begin{split} \label{eq:short_exp}
&S_{ij} \overline{S_{ij}} \\
&= R_{ij} \overline{R_{ij}} - R_{ia} V_{aa} R_{aj} \overline{R_{ij}} - R_{ij} \overline{R_{ia} V_{aa} R_{aj}} + R_{ia} V_{aa} R_{aj} \overline{R_{ia} V_{aa} R_{aj}} \\
& \quad + [(RV)^2 S]_{ij} \overline{R_{ij}} - [(RV)^2 S]_{ij} \overline{R_{ia} V_{aa} R_{aj}} + R_{ij} \overline{[(RV)^2 S]_{ij}} - R_{ia} V_{aa} R_{aj} \overline{[(RV)^2 S]_{ij}} + |[(RV)^2 S]_{ij}|^2.
\end{split} \end{equation}
Note that $V_{aa} \leq q^{-1}$ with high probability. When $i, j \neq a$, we have from the estimates \eqref{eq:S_bound} and \eqref{eq:R_bound} that
\begin{equation}
\big| S_{ij} \overline{S_{ij}} - R_{ij} \overline{R_{ij}} \big| \leq C N^{-1 + 6 \epsilon} q^{-1}
\end{equation}
with high probability. Let
\begin{equation} \label{eq:def_y^R}
y^S := \eta^2 \sum_{i \neq j} S_{ij} \overline{S_{ij}}, \quad y^R := \eta^2 \sum_{i \neq j} R_{ij} \overline{R_{ij}}.
\end{equation}
When $i = a$ or $j = a$, we have one less off-diagonal entries of $R$ or $S$ in the expansion \eqref{eq:short_exp}, but there are only $O(N)$ such terms. Thus, we obtain with high probability that
\begin{equation}
|y^S - y^R| \leq N^{-1/3 + 4 \epsilon} q^{-1}.
\end{equation}

Consider the Taylor expansion
\begin{equation}
F(y^S) - F(y^R) = F'(y^R) (y^S - y^R) + \frac{1}{2} F'' (\zeta) (y^S - y^R)^2
\end{equation}
for some $\zeta$, which lies between $y^S$ and $y^R$. Since $y^S \leq N^{2 \epsilon}$ and $y^R \leq N^{2 \epsilon}$ with high probability as we can see from the bounds \eqref{eq:S_bound} and \eqref{eq:R_bound}, $|F''(\zeta)| \leq N^{C \epsilon}$ with high probability from the assumption. Thus, we obtain
\begin{equation}
\left| \E \big( F(y^S) - F(y^R) \big) \right| \leq \left| \E \big( F'(y^R) \E_{aa} (y^S - y^R) \big) \right| + C N^{-2/3 + C\epsilon} q^{-2}. 
\end{equation}
For the first term of right hand side, we use \eqref{eq:short_exp} and the fact:  $R$ is independent of $V$ and $\E V_{aa}=0$ and bound this term  as $O(N^{-1/3 + 4\epsilon} q^{-2})$. Therefore, 
\begin{equation}
\left| \E \big( F(y^S) - F(y^R) \big) \right| \leq O(N^{-1/3 + 4\epsilon} q^{-2}).
\end{equation}

Note that we can get the same estimate if we put $T$ in place of $S$. Hence, we find that
\begin{equation} \label{eq:S-T 2}
\left| \E F \left( \eta^2 \sum_{i \neq j} S_{ij} \overline{S_{ij}} \right) - \E F \left( \eta^2 \sum_{i \neq j} T_{ij} \overline{T_{ij}} \right) \right| \leq C N^{-1/3 + C \epsilon} q^{-2}.
\end{equation}
We write the quantity in the left hand side of \eqref{eq:claim_old} as a telescopic sum,
\begin{equation} \label{eq:tel_1}
\left| \E F \left( \eta^2 \sum_{i \neq j} G^{\V}_{ij} \overline{G^{\V}_{ij}} \right) - \E F \left( \eta^2 \sum_{i \neq j} G^{\w}_{ij} \overline{G^{\w}_{ij}} \right) \right| = \sum_{\gamma=1}^{\gamma_{\max}} \left(\E F \left( \eta^2 \sum_{i \neq j} S_{ij} \overline{S_{ij}} \right) - \E F \left( \eta^2 \sum_{i \neq j} T_{ij} \overline{T_{ij}} \right)\right).
\end{equation}
Since the number of summands in the right hand side of \eqref{eq:tel_1} with $a \neq b$ is $O(N^2)$ and the number of summands with $a=b$ is $O(N)$, we find that \eqref{eq:claim_old} holds from \eqref{eq:S-T 1} and \eqref{eq:S-T 2}. This proves the claim, which implies the desired lemma.
\end{proof}

\section{Proof of the main result} \label{sec4}

In this section, we prove the main result, Theorem \ref{main result}. Let $H_N$ be a Wigner matrix defined as in Definition \ref{def: SWigner} such that $x_{12}$ and $x_{11}$ have distributions $\nu$ and $\wt \nu$, respectively. We begin by proving the second part of the main result.

\begin{proof}[Proof of the main result: Necessary condition]
Assume that $\lim_{s \to \infty} s^{4} {\p}(|x_{12}| \ge s) \neq 0$. We note that there exists a constant $0 < c_1 < 1/2$ and a sequence $r_1, r_2, \cdots$ such that $r_n \to \infty$ as $n \to \infty$ and 
\be
{\p}(|x_{12}| \ge r_k) \ge c_1 r_k^{-4}.
\ee

Consider an event
\be
\Gamma_N := \{ \text{ There exist } i \text{ and } j, \: 1 \leq i < j \leq N, \text{ such that } |h_{ij}| \geq 4, |h_{ii}| < 1, \text{ and } |h_{jj}| < 1 \}.
\ee
We first show that, when $\Gamma_N$ holds, $\lambda_N (H_N) \geq 3$. Define $\U \in \mathbb{R}^N$ through
\be
\U(m) =
	\begin{cases}
	1/\sqrt{2} & \text{ if } m = i, \\
	(1/\sqrt{2}) \cdot \text{sgn} (h_{ij}) & \text{ if } m = j, \\
	0 & \text{ otherwise }.
	\end{cases}
\ee
Here, $\text{sgn} (h_{ij}) := |h_{ij}| / h_{ij}$. Since $\| \U \|_2 = 1$, it can be easily seen that
\be
\lambda_N (H_N) \geq \langle \U, H_N \U \rangle = h_{ii} | \U_i |^2 + h_{jj} | \U_j |^2 + h_{ij} \U_i \U_j + h_{ji} \U_j \U_i = |h_{ij}| + \frac{h_{ii} + h_{jj}}{2} \geq 3.
\ee

We now prove that there exists a constant $c_0 > 0$, independent of $N$, such that $\p (\Gamma_N) \geq c_0$ for any $N \in \{ \lfloor (r_k / 4)^2 \rfloor : k \in \mathbb{N} \}$ with $N \geq 2 \E|x_{11}|^2$. Note that it implies \eqref{nec}. Define an event
\be
\Gamma'_N := \{ \text{ There exist } i \text{ and } j, \: 1 \leq i < j \leq N, \text{ such that } |h_{ij}| \geq 4 \}.
\ee
Clearly, if $N = \lfloor (r_k / 4)^2 \rfloor$, then we have that
\be \begin{split}
1 - \p (\Gamma'_N) &\leq \big( 1 - \p (|h_{12}| \geq 4) \big)^{N(N-1)/2} \leq \big( 1 - \p (|x_{12}| \geq r_k) \big)^{N(N-1)/2} \\
&\leq \left( 1 - c_1 r_k^{-4} \right)^{N(N-1)/2} \leq (1 - c_2 N^{-2})^{N(N-1)/2}
\end{split} \ee 
for some constant $0 < c_2 < 1$, independent of $N$. Since $(1 - c_2 N^{-2})^{N(N-1)/2} \leq c_3$ for some constant $0 < c_3 < 1$, independent of $N$, we find that $\p (\Gamma'_N) \geq 1- c_3$. Suppose that $\Gamma_N$ holds with $|h_{i' j'}| \geq 4$ for some indices $i'$ and $j'$ $(1 \leq i' < j' \leq N)$. From Markov inequality, we have
\be
\p ( |h_{i'i'}| > 1 ) = \p ( |x_{i'i'}| > \sqrt{N} ) \leq N^{-1} \E |x_{i'i'}|^2 \leq \frac{1}{2}
\ee
and $\p ( |h_{j'j'}| > 1 ) \leq 1/2$ as well. Since the diagonal elements $h_{i'i'}$ and $h_{j' j'}$ are independent to each other, and the event $\p ( |h_{i'i'}| \leq 1, |h_{j'j'}| \leq 1 )$ is indepedent from $\Gamma'_N$, we find that 
\be
{ \p (\Gamma_N) \geq \p (\Gamma'_N) (1/2)^2 \geq \frac{1}{4}(1-c_3).}
\ee
This completes the proof.
\end{proof}

Next, we prove that, if \eqref{crit} holds, then
\begin{equation} \label{suff}
\p \Big( N^{2/3} (\wt \lambda_N -2) \leq s - N^{-\delta} \Big) - N^{-\delta} \leq \p^{\V} \big( N^{2/3} (\lambda_N -2) \leq s \big) \leq \p \Big( N^{2/3} (\wt \lambda_N -2) \leq s + N^{-\delta} \Big) + N^{-\delta},
\end{equation}
where $\wt \lambda_N$ denotes the largest eigenvalue of GOE.

If $H_N$ satisfies the assumption in Theorem \ref{thm: new rigidity}, then \eqref{suff} indeed holds as we have seen from Theorem \ref{thm: new edge un}. Thus, we construct from $H_N$ a random matrix $H^S$ satisfying the bounded support condition and the moment condition in Theorem \ref{thm: new rigidity}. Comparing the largest eigenvalue of $H^S$ with that of $H_N$, we will show that the difference between them will be negligible with probability $1 - o(1)$. 

\begin{proof}[Proof of the main result: Sufficient condition]
For fixed $\e>0$, any $N$, define
\be \label{de61}
\al:=\al_N:=\p\left(|x_{12}|> N^{1/2-\e}\right), \quad \wt \al:=\wt\al_N:=\p\left(|x_{11}|> N^{1/2-\e}\right),
\ee
\be\label{de62}
\beta:= \beta_N:=\E \left[ {\bf 1}\left(|x_{12}|>  N^{1/2-\e}\right) x_{12} \right], \quad  \wt\beta:=\wt \beta_N:=\E \left[ {\bf 1}\left(|x_{11}|>  N^{1/2-\e}\right) x_{11} \right].
\ee
By \eqref{crit} and integration by parts, it implies that for any $\delta$ and large enough $N$, 
\be
\al \le \delta N^{-2+4\e},\quad \wt \al \le \delta N^{-1+2\e},
\ee
\be
|\beta|\le \delta N^{-3/2+3\e},\quad | \wt \beta|\le \delta N^{-1/2+\e}.
\ee

Let $\nu_S$, $\nu_L$ , $\wt \nu_S$, $\wt \nu_L$ have the distribution densities:
\be\label{95}
\rho_{\nu_S}(x+\beta)= {\bf} 1_{|x|\leq N^{1/2-\e}} \frac{\rho_\nu(x) }{1-\al_N}, \quad \rho_{\nu_L}(x { +\beta})= {\bf }1_{|x|> N^{1/2-\e}} \frac{\rho_\nu(x) }{ \al_N},
\ee
\be\label{96}
\rho_{\wt \nu_S}(x+\wt\beta)= {\bf} 1_{|x|\leq N^{1/2-\e}} \frac{\rho_{\wt\nu}(x) }{1-\al_N}, \quad \rho_{\wt \nu_L}(x { +\wt\beta} )= {\bf }1_{|x|> N^{1/2-\e}} \frac{\rho_{\wt\nu}(x) }{ \al_N}.
\ee
Here, the subindices $S$ and $L$ are for small and large, respectively. Let $\nu_c$, $\wt \nu_c$ be the distribution such that $x=1$ with probability $\al $ and $\wt \al$, otherwise $x=0$.  

Let $H^S$, $H^L$ and $H^C$ be the random matrices such that $H^S_{ij}=N^{-1/2} x^S_{ij}$, $H^L_{ij}=N^{-1/2} x^L_{ij}$, and $H^C_{ij}= c_{ij}$, where $x^S_{ij}$, $x^L_{ij}$, and $c_{ij}$ are independent random variables such that:

(1) the entries $x^S_{ij}$ have distribution $\nu_S$ if $i\ne j$ and  $x^S_{ii}$ have distribution $ \wt \nu_S$,

(2) the entries $x^L_{ij}$ have distribution $\nu_L$ if $i\ne j$ and  $x^L_{ii}$ have distribution $ \wt \nu_L$,  

(3) the entries $c_{ij}$ have distribution $\nu_c$ if $i\ne j$ and $c_{ii}$ have distribution $\wt \nu_c$. 
 
Clearly for independent  $H^S_{ij}$, $H^L_{ij}$ and $H^C_{ij}$, we know 
\be \label{de67}
H_{ij} \sim H^S_{ij} (1-H^C_{ij}) + H^L_{ij} H^C_{ij} + N^{-1/2} \beta + N^{-1/2} \wt \beta \delta_{ij},
\ee
where the notation ``$\sim$" denotes that the both sides have the same distribution. It is easy to see that the matrix $M$ defined by
$$
M_{ij}=N^{-1/2}\beta+N^{-1/2}\wt \beta \delta_{ij}
$$
satisfies that $\| M \|\leq N^{-1+4\e}$, which is negligible for small $\e$, i.e.,   
\be\label{jlxf}
\| H^S _{ij}(1-H^C_{ij})+H^L_{ij}H^C_{ij}-H\|\leq N^{-1+4\e}
\ee
By \eqref{crit} and integration by parts, we have for $i\neq j$, 
\be
\E (x^S_{ij})=0,  \quad \E |x^S_{ij}|^2=\frac{1}{1-\al }\int_{|x|\leq N^{1/2-\e}} {  (x+\beta)}^2 \rho_\nu(x) dx=1-O(N^{-1+2\e}),
\ee
$$
\E |x^S_{ij}|^3= O(1),  \quad \E  |x^S_{ij}|^4=   O (\log N) ,
$$
and 
\be
\E x^S_{ii} =0, \quad \E |x^S_{ii}|^2=\frac{1}{1-\wt \al}\int_{|x|\leq N^{1/2-\e}} {  (x+\wt\beta)}^2 \rho_\nu(x) dx=1-o(1).
\ee
We note that the matrix
$$
\wt H^S:=\frac{1}{\sqrt{\E |x^S_{12}|^2}} H^S 
$$
is a Wigner matrix that satisfies the assumption of $H^ \V$ in Theorem \ref{thm: new edge un} ($H $ in Theorem \ref{thm: new rigidity}). Together with the fact $\E |x^S_{12}|^2 = 1-O(N^{-1+2\e})$, we find that there exists a $\delta > 0$ such that, for any $s \in \mathbb{R}$, we have
\begin{equation}\label{qbtt}
\p^{S} \Big( N^{2/3} (\lambda_N -2) \leq s - N^{-\delta} \Big) - N^{-\delta} \leq \p^{GOE} \big( N^{2/3} (\lambda_N -2) \leq s \big) \leq \p^{S} \Big( N^{2/3} (\lambda_N -2) \leq s + N^{-\delta} \Big) + N^{-\delta}.
\end{equation}
where $\p^S$ is the law for $H^S$. We write the first two terms in the right hand side of \eqref{de67} as follows:
$$
H^S_{ij}(1-H^C_{ij}) + H^L_{ij} H^C_{ij}= H^S_{ij} + \wt H_{ij} H^C_{ij}, \quad \wt H_{ij}= H^L_{ij} - H^S_{ij} \, .
$$ 
We can see that $H^C_{ij}$ is independent of $H^S_{ij}$ and $\wt H_{ij}$. Though $\wt H_{ij}$ depends on $H^S_{ij}$, from the condition \eqref{crit} we know that for any $i, j$ 
\be
\p(|\wt H_{ij} H^C_{ij} |\ge 3/4)\le \p(|x_{ij}| \ge (1/2) N^{1/2}) \leq o( N^{-2})+o ( N^{-1}) \delta_{ij}.
\ee
Here, for the last inequality, we used \eqref{crit} and that $\E|x_{ii}|^2<\infty$. Because of this reason, instead of $H^S _{ij} + \wt H_{ij} H^C_{ij}$, we only need to study the matrix whose entries have a cutoff on $\wt H_{ij} H^C_{ij}$ as follows: 
\be
H^S _{ij} + \wt D_{ij} H^C_{ij}, \quad   \wt D_{ij}=   {\bf 1}_{\wt H_{ij}H^C_{ij}\leq 3/4} \wt H_{ij}.
\ee
We note that
\be \label{mkgd0}
\p \left(\max_{i, j}|\wt D_{ij} H^C_{ij}-\wt H_{ij} H^C_{ij}|=0  \right)=1-o(1).
\ee
Furthermore, it is easy to see that we can introduce a cutoff on matrix $H^C$ such that:
\begin{itemize}
\item The matrix with the cutoff, $\wt{H}^C$ coincides with $H^C$ with probability higher than $1 - o(1)$, i.e.,
\be\label{mkgd}
\p( \wt{H}^C = H^C ) \ge 1 - o(1)
\ee

\item The number of non-zero entries are bounded by
\be
\#\{ (i,j): \wt H^C_{ij}\neq 0   \}\leq N^{5\e}
\ee
\item If $\wt{H}^C_{ij}\neq 0$ and $\wt{H}^C_{kl}\neq 0$, then either $\{i,j\}=\{k,l\}$ or $\{i,j\}\cap \{ k,l \}= \emptyset$.
\end{itemize}
 
With \eqref{mkgd}, we only need to study the largest eigenvalue of 
$$
H^S  +  E  \quad {\rm where} \quad E_{ij}:=  \wt D_{ij}\wt H^C_{ij}.
$$
We note 
 $\max_{ij}|E_{ij}|\leq 3/4$, and the rank of $E$ is less than $N^{5\e}$.  

Let $\lambda_N$ and $\mu_N$ be the largest eigenvalue of $H^S$ and $H^S + E$, respectively. We claim that  
\be\label{segg}
\p\left(|\lambda_N-\mu_N|\leq N^{-3/4}\right) = 1-o(1).
\ee
From the claim \eqref{segg} together with \eqref{qbtt}, \eqref{mkgd}, \eqref{mkgd0} and \eqref{jlxf}, we obtain the desired result \eqref{def: ranvec}.

Now we prove \eqref{segg}. Recall $\wt{H}^C$ and $H^S$ are independent, i.e.,  the positions of the nonzero elements of $E$ is independent of $H^S$. Then by symmetry, we can assume that for some $s, t\leq N^{5\e}$, among the matrix entries of $E$, only 
$$
E_{12}, \quad E_{21}, \quad E_{34}, \quad E_{43}\cdots, E_{2s-1,2s}, \quad E_{2s,2s-1}, \quad  E_{2s+1, 2s+1}, \quad  E_{2s+2, 2s+2}, \cdots  E_{2s+t, 2s+t}
$$
are non-zero. Then, we can decompose the $E$ as 
\be
E=VDV^*
\ee
where $D$ is a $(2s+t)\times (2s+t)$ diagonal matrix and $V$ is $N \times (2s+t)$. Furthermore, 
\be
D= {\rm diag}\{E_{12}, -E_{12},E_{34}, -E_{34},\cdots, E_{2s-1, 2s}, -E_{2s-1,2s}, E_{2s+1, 2s+1},  E_{2s+2, 2s+2},  \cdots  E_{2s+t, 2s+t} \}.
\ee
If $j=2n-1$, $n\le s$, then 
\be
V_{ij}=\delta_{i,2n-1}\frac{1}{\sqrt 2}+\delta_{i,2n }\frac{1}{\sqrt 2}
\ee
If $j=2n $, $n\le s$,  then 
\be
V_{ij}=\delta_{i,2n-1}\frac{1}{\sqrt 2}-\delta_{i,2n }\frac{1}{\sqrt 2}
\ee
If $j=2s+n $, $1\le n\le t$,  then 
\be
V_{ij}=\delta_{ij} 
\ee
Note that $E$ is symmetric matrix. Using Lemma 6.1 in \cite{KY}, we find that, if $\mu$ is the eigenvalue of $H^S+E$, then 
\be\label{tyyz}
\det \left(V^* G^S(\mu)V+D^{-1}\right)=0 \quad {\rm where} \quad G^S(\mu)= \frac{1}{H^S-\mu}
\ee
Similarly, if we let $\mu^\gamma$ be the eigenvalue of $H^S+\gamma E$, then 
\be
\det \left(V^* G^S(\mu)V+(\gamma D)^{-1} \right) =0 \quad {\rm where} \quad G^S(\mu)= \frac{1}{H^S-\mu}.
\ee

Define $(2s+t)\times (2s+t)$ matrix $F$ by
\be
F^\gamma := F^\gamma(\mu) := V^* G^S(\mu)V+(\gamma D)^{-1}.
\ee
Then, we have for the following $2\times 2$ blocks of $F^\gamma$, 
\be\label{FG1} \begin{split}
 &\left( \begin{matrix} F^\gamma_{2\al-1, 2\beta-1}&F^\gamma_{2\al-1, 2\beta}\\ F^\gamma_{2\al, 2\beta-1}&F^\gamma_{2\al, 2\beta } \end{matrix} \right)
 \\
 &=\frac{1}{2} \left( \begin{matrix} 1&1\\ 1&-1\end{matrix} \right) 
 \left( \begin{matrix} G^S_{2\al-1, 2\beta-1}&G^S_{2\al-1, 2\beta}\\ G^S_{2\al, 2\beta-1}&G^S_{2\al, 2\beta } \end{matrix} \right)
  \left( \begin{matrix} 1&1\\ 1&-1\end{matrix} \right)
  +\delta_{\al, \beta} \left( \begin{matrix} (\gamma E_{2\al-1, 2\al})^{-1}&0\\0 &-(\gamma E_{2\al-1, 2\al})^{-1}\end{matrix} \right)
\end{split} \ee
where $1\le \al, \beta\le s$. For the following  $2\times 1$ blocks of $F$, we get
\be\label{FG2}
 \left( \begin{matrix} F^\gamma_{2\al-1, \, \beta } \\ F^\gamma_{2\al,  \, \beta }  \end{matrix} \right)
 =\frac{1}{\sqrt 2} \left( \begin{matrix} 1&1\\ 1&-1\end{matrix} \right) 
 \left( \begin{matrix} G^S_{2\al-1, \, \beta } \\ G^S_{2\al,  \, \beta }  \end{matrix} \right)
 \ee
where $1\le \al \le s$, $2s+1\le \beta\le 2s+t$. Finally, for the following  $1\times 1$ blocks of $F$, 
\be\label{FG3}
 \left( \begin{matrix} F^\gamma_{ \al , \, \al }   \end{matrix} \right)
 = \left( \begin{matrix} G^S_{ \al , \, \al }   \end{matrix} \right)+(\gamma E_{\al, \al })^{-1}
  \ee
where  $2s+1\le \al \le 2s+t$. 

Let $\wt z=2+iN^{-2/3}$. From \eqref{Gij estimate} and \eqref{smallz}, with high probability, we have 
\be\label{928o}
\max_{1\leq i \leq 2s+t}|G^S_{ii}(\wt z)+1 |\leq N^{-\e/2}.
\ee
For off-diagonal terms, with \eqref{res: WBGij} and \eqref{esmallfake}, we have that
\be\label{928}
\max_{1\leq i\neq j\leq 2s+t}|G^S_{ij}(\wt z) |\leq N^{-1/6}
\ee
holds with probability $1-o(1)$. Define
\be
\wt\mu:=\lambda_N \pm N^{-3/4}.
\ee
From remark \eqref{genedge} and the fact that the largest eigenvalues of GOE are separated in the scale $N^{-2/3}$, we have 
\be
\p(\lambda_N-\lambda_{N-1}\ge 2N^{-3/4})\ge 1-o(1),
\ee
thus
\be\label{ttg}
\p(\min_{\al}|\lambda_\al- \wt\mu | \ge N^{-3/4})\ge 1-o(1).
\ee
With the distribution of $\lambda_N$ (see \eqref{qset1}), we also have with probability $1-o(1)$ that
\be\label{ttg2}
|\wt \mu-2|\leq N^{-2/3}\varphi^C.
\ee
Consider the identity
\be \begin{split}
\left|G^S_{ij}(\wt z)-G^S_{ij}(\wt\mu)\right|
&=\sum_\al \left|u_\al(i)u_\al(j)\right|\left|\frac{1}{\lambda_\al-\wt z}-\frac{1}{\lambda_\al-\wt\mu}\right|\\
&=\sum_\al \left|u_\al(i)u_\al(j)\right|\left(\frac{|\wt\mu-\wt z|}{|\lambda_\al-\wt z||\lambda_\al-\wt\mu|} \right).
\end{split} \ee
From \eqref{ttg}, \eqref{ttg2}, the delocalization in Lemma \ref{lem: old}, and the rigidity result in Theorem \ref{thm: new rigidity}, we have with probability $1-o(1)$ that 
\be\label{2ss}
\max_{i,j}| G^S_{ij}(\wt z)-G^S_{ij}(\wt \mu)|\leq CN^{-1/4}.
\ee
Combining \eqref{928o}, \eqref{928}, and \eqref{2ss}, we find that \eqref{928o} and \eqref{928} still hold with $\wt z$ replaced by $\wt \mu$ and the right hand side being doubled. From that $\max_{i,j}| E_{ij}|\leq 3/4$, together with \eqref{FG1}-\eqref{FG3}, we have for any $0 < \gamma \leq 1$ that
\be
\min_{i,\gamma} |F^\gamma_{ii}(\wt\mu)|\ge \frac13-  O(N^{-\e/2}), \quad \max_{i,j, \gamma} |F^\gamma_{ij}(\wt\mu)|\le  CN^{-1/6}+{\bf 1}_{|i-j|=1}CN^{-\e/2}
\ee
holds with probability $1-o(1)$. This implies that, since $\e$ is small enough, 
$$
\det F^\gamma(\wt\mu)\neq 0.
$$
Recall \eqref{tyyz}, then we know that  the following event holds with probability $1-o(1)$:  $\wt\mu$ is not the eigenvalue of $H^S+\gamma E$ for any $0 < \gamma \leq 1$. 

If we let $\mu^\gamma_N$ be the the largest eigenvalue of $H^S+\gamma E$ for $0 \leq \gamma \leq 1$, then by definition $\lambda_N=\mu^0_N$, since $\lambda_N$ is the largest eigenvalue of $H^S$. With the continuity of $\mu^\gamma_N$ with respect to the $\gamma$, we find that $\wt\mu$ is not the eigenvalue of $H^S$, hence we have that, for any $0 \leq \gamma \leq 1$,
\begin{equation}
\mu^\gamma_N \in (\lambda_N - N^{-3/4}, \lambda_N + N^{-3/4} )
\end{equation}
holds with probability $1-o(1)$. Thus, we have proved \eqref{segg}, which implies the desired result \eqref{def: ranvec}.   
\end{proof}

\section{Basic ideas for Theorem \ref{thm: new rigidity} and Theorem \ref{thm: new edge un}} \label{sec6}

The basic idea of proving Theorem \ref{thm: new rigidity} and Theorem \ref{thm: new edge un} is Green function comparison method, as mentioned in the introduction. To apply the comparison results, first we show that for any $H$ in Theorem \ref{thm: new rigidity}, there exists a matrix $\wt H$ whose entries have the same first four moments as those of $H$ and satisfies the bounded support condition with large $q=O(N^{1/2}/\log N)$. Roughly speaking, the $\wt H$ has the properties that we need to prove for $H$ and we will use Green function comparison method to show that $H$ has the same properties, since $H$ and $\wt H$ have the same first four moments.

\begin{lem}\label{lem: match}
For any generalized Wigner matrix $H$ under the assumptions of Theorem \ref{thm: new rigidity}, there exists another generalized Wigner matrix $\wt H$, such that $\wt H$ satisfies bounded support condition with $q=O(N^{1/2}/\log N)$ and the first four moments of the off-diagonal entries of $H$ and $\wt H$ match, i.e.,
\be
\E x^k_{ij}= \E \wt x^k_{ij} \; {  (i \neq j)}, \quad k=1,2,3,4
\ee
and the first two moments of the diagonal entries of $H$ and $\wt H$ match, i.e., $\E x^k_{ii}= \E \wt x^k_{ii}, \quad k=1,2$.
\end{lem}

\begin{proof}[Proof of Lemma \ref{lem: match}]
The diagonal part is trivial. For the off-diagonal part, it clearly follows from the next lemma.
\end{proof}

\begin{lem}
For any $C>0$, if $C \le |A|$ and $B\ge A^2+1$, there exists a random variable $x$ such that 
\begin{equation}\label{tyss}
\E(x)=0, \quad \E(x^2)=1, \quad \E(x^3)=C, \quad \E(x^4)=B,
\end{equation}
and 
\begin{equation}\label{tyss2}
{\rm  supp(x)} \subset [-DB, DB]
\end {equation}
for some $D$ depending only on $C$. 
\end{lem}

\begin{rem}
The condition $B \ge A^2+1$ comes from the simple fact that if $\E(x)=0$ and $\E(x^2)=1$ then $\E (x^4) \ge 1+ (\E (x^3) )^2$.
\end{rem}

\begin{proof}
For fixed $A$ and any $B \leq k$, it is easy to find a distribution such that \eqref{tyss} and \eqref{tyss2} hold with $D$ depending on $k$. Therefore, one only needs to show that this lemma holds in the case that $B$ is large enough. To show that, first we introduce a family of random variables $Y_t (t \geq 1)$, whose distribution has a finite support
\begin{equation} \label{defYt}
{\rm supp(Y_t)} \in \left\{ -t, \frac{-\sqrt{t}}{\sqrt{1+t}}, \frac{\sqrt{t}}{\sqrt{1+t}}, t \right\}
\end {equation}
and satisfies
\be
\p(Y_t= t) = \p(Y_t= -t) = \frac{1}{2 t (-1 + t + t^2)} \, ,
\ee
\be
\p(Y_t= \frac{\sqrt{t}}{\sqrt{1+t}}) = \p(Y_t= \frac{-\sqrt{t}}{\sqrt{1+t}}) = \frac{1}{2} - \frac{1}{2 t (-1 + t + t^2)} \,.
\ee

One can easily check that every odd moment of $Y_t$ vanishes and 
\begin{equation} 
\E (Y_t)^2 = 1, \quad \E (Y_t)^4 = t.
\end {equation}
Note that $Y_t$ is supported in $[- t, t]$.

We choose another random variable $X$ whose distribution is supported on 
$$
\{ \sqrt{2}A - \sqrt{1 + 2A^2}, \sqrt{2}A + \sqrt{1 + 2A^2} \}
$$
and satisfies
\be
\p(X = \sqrt{2}A - \sqrt{1 + 2A^2}) = \frac{\sqrt{2}A + \sqrt{1 + 2A^2}}{2 \sqrt{1 + 2A^2}}, 
\ee
\be
\quad \p(X = \sqrt{2}A + \sqrt{1 + 2A^2}) = \frac{-\sqrt{2}A + \sqrt{1 + 2A^2}}{2 \sqrt{1 + 2A^2}}.
\ee
Then, simple calculation shows that
$$
\E X=0, \quad \E X^2=1, \quad \E X^3=2\sqrt 2A, \quad \E X^4= 8A^2 + 1.
$$

Choose
$$
t= 4B-8A^2-2
$$
and let
$$
x= X+Y_t \, .
$$
Since $X$ and $Y_t$ are independent, it can be easily check that $x$ satisfies \eqref{tyss} and \eqref{tyss2}, especially, for large enough $B$, ${\rm  supp(x)} \subset [-5B, 5B]$.
\end{proof}

Since $\wt H$ has the property we need (see Lemma \ref{lem: old}), we now compare $H$ with $\wt H$ using Green function comparison method.

To prove Theorem \ref{thm: new rigidity} and Theorem \ref{thm: new edge un}, we claim the following two lemmas first, which will be proved in the next section.

\begin{lem} \label{lem:p moment}
Let $H$ and $\wt H$ satisfy the assumptions of Lemma \ref{lem: match}. For $z\in S(C) $ in \eqref{defS} with large enough $C>0$, if for deterministic number  $X$ and $Y$, 
\be\label{asXY}
\max_{i\ne j}|\wt G_{ij}(z)|\leq X, \quad |\wt m-m_{sc}| \le Y
\ee
holds with $3$-high probability(see Def. \ref{def: hprob}), then for any $p\in 2\mathbb N$ with $p\leq \varphi$, we have 
\be\label{sqbm}
\E |m- m_{sc}|^p \leq \E | \wt m- m_{sc}|^p + (C p)^{Cp} \left(X^2+Y+\varphi N^{-1}\right)^{p}.
\ee
\end{lem}

\begin{lem} \label{lem:green}
Let $H$ and $\wt H$ satisfy the assumptions of Lemma \ref{lem: match}. Let $F: \mathbb{R} \to \mathbb{R}$ be a function whose derivatives satisfy
\begin{equation}\label{jhw2}
\sup_x | F^{(\al)} (x) | (1+ |x|)^{-C_1} \leq C_1, \quad \al = 1, 2, 3,
\end{equation}
with some constant $C_1 > 0$. Then, there exists a constant $\widetilde{\epsilon} > 0$, depending only on $C_1$, such that, for any $\epsilon < \widetilde{\epsilon}$ and for any real numbers
\begin{equation}\label{qglm}
E, E_1, E_2 \in \{ x : |x-2| \leq N^{-2/3 + \epsilon} \} =: I_{\epsilon},
\end{equation}
and setting $\eta := N^{-2/3 - \epsilon}$, then
\begin{equation}\label{mainpear}
\left| \E \: F \left( \eta^2 \sum_{i, j} G _{ij} \overline{G _{ij}} (z)\right)
 - \E \: F \left( \eta^2 \sum_{i, j} \wt G _{ij} \overline{\wt G _{ij}} (z)\right) \right| \leq N^{-\phi + C\epsilon}, \quad z=E+i\eta
\end{equation}
and for $E_1, E_2\in [2-N^{-2/3+\e}, 2+N^{-2/3+\e}]$, 
\begin{equation}\label{sbc22}
\left| \E F \left( N \int_{E_1}^{E_2} dy \: \im m (y + i \eta) \right) - \E  F \left( N \int_{E_1}^{E_2} dy \: \im \wt m (y + i \eta) \right) \right| \leq N^{-\phi + C \epsilon}.
\end{equation}
\end{lem}

\begin{proof}[Proof of Lemma \ref{thm: new rigidity}] 
Recall that, with \eqref{sjsy}, we have that for some large $C$, with high probability, 
\be\label{tby}
H\le 2+\varphi^{C}(N^{-2/3}+q^{-2}).
\ee
First, we improve this result to that 
\be\label{tby2}
H\le 2+\varphi^{C}N^{-2/3}
\ee
holds with high probablity. Let $\wt H$ match $H$ in the sense of Lemma \ref{lem: match}. For $a \ge 12$, with \eqref{3737} and \eqref{3838}, for $z=E+i\eta$ satisfying
\be \label{3737T}
2+\varphi^CN^{-2/3}\leq E\leq2+q^{-1}, \quad \eta=\varphi^{(3+a)} N^{-1}\kappa^{-1/2},
\ee
we have with 3-high probability
\be\label{3838T}
|\wt m-m_{sc}|\leq \frac{1}{\varphi^aN\eta }, \quad \im \wt m(z)\leq \frac{1}{\varphi^aN\eta },  \quad \max_{i\ne j}|\wt G_{ij}|\leq \frac{1}{\varphi^{a/3} N \eta}.
\ee
Assume that $C$ in \eqref{3737T} is greater than $8+4a$. Then, $\varphi^{2+a}\eta\leq \kappa$. With the property of $m_{sc}$ in \eqref{esmallfake}, we have $\im m_{sc}\sim \eta/\sqrt\kappa$, which implies 
\be\label{papk}
\im m_{sc}(z)\leq \frac{1}{\varphi N\eta}.
\ee
Now, we apply Lemma \ref{lem:p moment} on $H$ and $\wt H$ with $z$ in \eqref{3737T}, $p=\varphi$, $X=(\varphi^{a/3} N \eta)^{-1}$ and $Y=(\varphi^{a} N \eta)^{-1}$. Then, with \eqref{sqbm} and Markov inequality, for some $C>0$, we obtain with high probability that
\be\label{stmp}
|m- m_{sc}|\leq \varphi^C\left(  \frac{1}{\varphi^aN\eta }+  \frac{1}{\varphi^{2a/3}(N\eta)^2 } + N^{-1}\right).
\ee
From \eqref{3737T}, we know 
\be
\varphi^{O(1)} N^{1/3}\leq (N\eta)^{-1}\leq \varphi^{O(1)}N^{-\phi /2}.
\ee
Inserting it into \eqref{stmp} and choosing $a=C+1$ in \eqref{stmp}, we obtain $ |m- m_{sc}|\ll (N\eta)^{-1}$. With \eqref{papk}, it implies that 
$\im m(z)\ll (N\eta)^{-1}$ holds with high probability. By definition, 
$$
\im m= N^{-1} \sum_\al \eta ((\lambda_\al-E)^2+\eta^2)^{-1}.
$$
Then $\im m(z)\ll (N\eta)^{-1}$ implies that there are no eigenvalues in the interval $[E-\eta, E+\eta]$. Since it holds for any $z$ in \eqref{3737T} with high probability, we obtain that there are no eigenvalues in $[2+\varphi^C N^{-2/3}, 2+q^{-1}]$. Together with \eqref{tby}, we obtain \eqref{tby2}. By symmetry, we have 
\be\label{tby3}
\|H\|\le 2+\varphi^{C}N^{-2/3}.
\ee

Next, we apply Lemma \ref{lem:p moment} again on $H$ and $\wt H$ with $z$ in \eqref{scm}-\eqref{Gij estimate}, $p=\varphi$ and 
$$
X= \varphi^C \sqrt{\frac{\im m_{sc}}{N\eta}}+\frac{\varphi^C}{  N \eta}, \quad Y=\frac{\varphi^C}{  N \eta}.
$$
where $X, Y$ follows from \eqref{scm}-\eqref{Gij estimate}. Then, with \eqref{sqbm} and Markov inequality, we have that for some $C=O(1)$,  
\be\label{tby4}
\max_{z\in S(C)}|m(z)-m_{sc}(z)|\leq \varphi^C(N\eta)^{-1}
\ee 
holds with high probability. Following the argument of section 5 in \cite{EYY}, which was used to prove the (2.25) and (2.26) in \cite{EYY}, we obtain Lemma \ref{thm: new rigidity}. Note that we can almost take the varbatim except some coeffiencts, where $\varphi^C$ plays the role of $T_N$ there.  
\end{proof}

\begin{proof}[Proof of Theorem \ref{thm: new edge un}]
We first recall the following lemma which is basically proved in \cite{EYY}.

\begin{lem}\label{lem: edge un str}
Suppose that two generalized Wigner matrices $H^{\w}$ and $H^\V$ satisfy with high probability that \eqref{rigiditynew}, \eqref{rigiditynew2}, \eqref{tby4}, and
\begin{equation} \label{hh1}
|N^{2/3} (\lambda_N - 2) | \leq \varphi^C
\end{equation}
and the number of eigenvalues in $[2 -  2 \varphi^C N^{-2/3}, 2 +  2 \varphi^C N^{-2/3}]$ is bounded as follows:
\begin{equation}\label{hh2}
\N \left( 2 - \frac{2 \varphi^C}{N^{2/3}}, 2 + \frac{2 \varphi^C}{N^{2/3}} \right) \leq \varphi^{2C} 
\end{equation}
for some constant $C$. If, moreover, they satisfy the conditions from \eqref{jhw2} to \eqref{sbc22}, then there exists a constant $\delta$ such that, for any $s>0$,
\begin{equation} \label{qset15}
\p^{\w} \Big( N^{2/3} (\lambda_N -2) \leq s - N^{-\delta} \Big) - N^{-\delta} \leq \p^{\V} \big( N^{2/3} (\lambda_N -2) \leq s \big) \leq \p^{\w} \Big( N^{2/3} (\lambda_N -2) \leq s + N^{-\delta} \Big) + N^{-\delta}.
\end{equation}
\end{lem}

\begin{proof}[Proof of Lemma \ref{lem: edge un str}]
Though this lemma is not explicitly stated in \cite{EYY}, this is the basic structure of proving the edge universality, Theorem 2.4, in \cite{EYY}. In the section 6 of \cite{EYY}, the edge universality problem was converted into proving Theorem 6.3 in \cite{EYY}. The conditions from \eqref{jhw2} to \eqref{sbc22} in this paper are exactly the same as Theorem 6.3 in \cite{EYY}.  To obtain this conversion, in section 6 of \cite{EYY}, only the assumptions in Lemma \ref{lem: edge un str} of this paper was used. To help readers compare, we note that they are (2.19), (2.25), (2.26), (6.2), and (6.3) in \cite{EYY}.
\end{proof}
 
Now, we return to prove Theorem \ref{thm: new edge un}. We apply Lemma \ref{lem: edge un str} with $H=H^{\bf v}$ and $\wt H=H^{\bf w}$. Clearly, it only remains to check \eqref{hh1} and \eqref{hh2}. One can see that it follows from \eqref{rigiditynew}, the rigidity of eigenvalues, and that $\gamma_N=2$. Then, with Lemma \ref{lem: edge un str}, 
\begin{equation}  
\p \Big( N^{2/3} (\wt\lambda_N -2) \leq s - N^{-\delta} \Big) - N^{-\delta} 
\leq \p \big( N^{2/3} (\lambda_N -2) \leq s \big)
\leq \p \Big( N^{2/3} (\wt\lambda_N -2) \leq s + N^{-\delta} \Big) + N^{-\delta},
\end{equation}
where $\wt\lambda_N$ denotes the largest eigenvalue of $\wt H$. Furthermore, since $\wt H$ satisfies the bounded support condition with large $q\sim N^{1/2}/\log N$, with Lemma \ref{lem: ex edge un}, we obtain \eqref{qset1} and complete the proof of Lemma \ref{thm: new edge un}.  
\end{proof}

\section{Proof of Lemma \ref{lem:p moment}, Lemma \ref{lem:green}, and Lemma \ref{lem: WBGij}} \label{sec7}
 
To prove Lemmas \ref{lem:p moment} and \ref{lem:green}, we will again use Green function comparison method. Recall the notations in \eqref{eq:ordering}-\eqref{eq:green_fns} with $\wt H=H_0 = H^{\w}$ and $H=H_{\gamma_{\max}} = H^{\V}$. We let $S=G^\gamma=(H_\gamma-z)^{-1}$, $R=(Q-z)^{-1}$, and $T=G^{\gamma-1}=(H_{\gamma-1}-z)^{-1}$, where $Q$ depends on $\gamma = \Phi(a,b)$. Note that $H_\gamma$ satisfies the bounded support condition with $q=N^{-\phi}$ for all $1\leq \gamma \le \gamma_{\max}$.

From Lemma \ref{lem: old} and \eqref{8800}, we know that there exists a uniform constant $C_0>0$ such that, with $3-$high probability,
\be\label{813}
\max_{z\in S(C_0)}\max_\gamma \max_{i,j}\max\{|S_{ij}(z)|, |R_{ij}(z)|, |T_{ij}(z)|\}\leq 2, 
\ee 
where we used the bound $|m_{sc}|\leq 1$. We note that the uniformity here is easy to check, since there are only finitely many different distributions for all the matrix entries of $H_\gamma$, $1\leq \gamma\leq \gamma_{\max}$. On the other hand, $S$, $R$, and $T$ satisfy the following trivial bound that always holds:
\be\label{813tr}
\max_{z\in S(C_0)}\max_\gamma \max_{i,j}\max\{|S_{ij}(z)|, |R_{ij}(z)|, |T_{ij}(z)|\}\leq \eta^{ -1} \leq N  .
\ee
In this section, for simplicity, we use the notation $|{\bf k}|=\|{\bf k}\|_1$ for any vector $k=\mathbb R^n$, $n\in \mathbb N$.

To illustrate the idea of Green function comparison method, we first consider the following simple example of finding a bound on $\E G_{ij}$ from an a priori bound on $\wt G_{ij}$:

\begin{ex}
Suppose that the bound
\be \label{eq:apriori}
\max_{i \neq j} |\wt G_{ij}| \leq N^{-1/3}
\ee
holds. Then, we have that $|\E G_{ij}| \leq C N^{-1/3}$.
\end{ex}

Applying the replacement strategy, we obtain
\be\label{gzb}
\E G_{ij} = \E \wt G_{ij} + \sum_{\gamma = 1}^{\gamma_{\max}} \left( \E G^{\gamma}_{ij} - \E G^{\gamma -1}_{ij} \right).
\ee
Note that the bound \eqref{eq:apriori} implies that $|\E \wt G_{ij}| \leq N^{-1/3}$. Thus, if we can prove that for $1\leq \gamma\leq \gamma_{\max}$
\be
\left| \E S_{ij} - \E T_{ij} \right| \leq C N^{-7/3},
\ee
then this will show the desired estimate on $|\E G_{ij}|$.

We now expand $S_{ij}$ in terms of $R$ and $V$, as in \eqref{8800},  using the resolvent expansion
$$
S = R - RVR + (RV)^2 R - (RV)^3 R + (RV)^4 R + \cdots + (-1)^m (RV)^m R + (-1)^{m+1} (RV)^{m+1} S,
$$
with $m := \lceil 1 / \phi \rceil$. Since each element of $V$ is bounded by $N^{-\phi}$ with high probability, from \eqref{813}, we find that the last term in the expansion is $O(N^{-1})$. Taking expectation, we find that
\be
\E S_{ij} = \E \sum_{k=0}^m [(RV)^k R]_{ij} + O(N^{-1}).
\ee
Note that $R$ is independent of $V$. We can decouple $R$ and $V$ by taking partial expectation $\E_{ab}$ with respect to $V_{ab}$, which gives
\be
\E S_{ij} = \E \sum_{k=0}^4 A_k \E_{ab} (V_{ab}^k) + O(N^{-1}),
\ee
where $A_k$ depends only on $R$. For example, $A_5$ contains a term such as $R_{ia} R_{bb} R_{aa} R_{bb} R_{aa} R_{bj}$. 

From the moment matching condition, we know that the first four moments of $H$ and $\wt H$ coincide, thus the terms with $k=0, 1, \cdots, 4$ will vanish when we estimate $\E S_{ij} - \E T_{ij}$. Moreover, since we know that $|\E (V_{ab}^k)| = O(N^{-2 - \phi/2})$, it suffices to prove the bound $|\E A_k| \leq C N^{-1/3 + \phi /2}$ for $5 \leq k \leq m$. 
Now using an expansion such as (see \eqref{eq:S_exp4})
$$
\E \prod_{t=1}^s R_{i_t j_t} = \E \prod_{t=1}^s S_{i_t j_t} - \E \sum_{l=1}^m \wt A_l \E_{ab} (V_{ab}^l) + O(N^{-1}),
$$
where $\wt A_l$ are products of $S$,  we write the expectation of the product of the elements of $R$'s (like $A_k$)  as the sum of the expectation of the product of the elements of $S$'s. Hence, we can convert the problem into showing
\be
|\E B_k| \leq C N^{-1/3 + \phi /2}, \quad k \geq 5,
\ee
where $B_k$ is a sum of the product of the elements of $S$. For example, $B_5$ contains a term such as $S_{ia} S_{bb} S_{aa} S_{bb} S_{aa} S_{bj}$. Note here we effectively gain a factor of $N^{-\phi/2}$ in the required bound.

We now repeat the replacement argument with the terms in $B_k$, since the matrix $H_{\gamma}$ also satisfies the same bounded support condition as $H$. We consider a telescopic series, as in \eqref{gzb},
\be\label{gzb2}
\E B_k = \E B_k (S \to \wt G) + \sum_{\gamma' = 1}^{\gamma} \left( \E B_k (S \to G^{\gamma'}) - \E B_k (S \to G^{\gamma '-1}) \right),
\ee
where the notation $B_k (S \to G^{\gamma'})$ means that we consider the product of $G^{\gamma'}=(H_{\gamma'}-z)^{-1}$ instead of $S$ while keeping the indices the same. The first term in the r.h.s. of \eqref{gzb2} will be the sum of the products such as $\E\wt G_{ia} \wt G_{bb} \wt G_{aa} \wt G_{bb} \wt G_{aa} \wt G_{bj}$, and in a generic case, it contains at least one off-diagonal term of $\wt G$. In particular, we have $|\E B_k (S \to \wt G)| \leq C N^{-1/3}$. Hence, we are left to estimate the telescopic sum, where we use the argument above, which involves the resolvent expansion and partial expectation, again. Note that, each time we repeat the procedure, we effectively gain a factor of $N^{-\phi/2}$ in the required bound. Therefore, after repeating the procedure sufficiently many times (i.e., $O(1/\phi)$-times), we find that it suffices to prove that $|\E \wt B| \leq C$, where $\wt B$ is a sum of the products of the elements of $G^{\gamma}$, $1\leq \gamma\le \gamma_{\max}$. Since this is trivial from the bound \eqref{813}, we find that the bound $|\E G_{ij}| \leq C N^{-1/3}$ holds.

We now prove the lemmas using the ideas explained above. We first introduce some notations for simplifying the expressions, which will helps us study the expectation of the product of $S_{ij}$'s. 

\begin{defn}[Matrix operators $*_\gamma$ and $*$] \label{def*} 
For a $\gamma$ with $\Phi(a,b)=\gamma$, we define $A*_\gamma B$ as
\be
(A*_\gamma B)_{ij} =(A{\mathcal I_\gamma} B)_{ij}, \quad\quad  ({\mathcal I_\gamma})_{ij}={\bf 1}_{\{i,j\}=\{a,b\}}
\ee
When $a \neq b$, it satisfies
\be
(A*_\gamma B)_{ij} = A_{ia}B_{bj}+A_{ib}B_{aj}.
\ee
We will often drop the subscript $\gamma$ for convenience as $A*  B$. For simplicity, we denote the $k$-th power of $A$ under $*-$ product by $A^{*k}$, i.e., 
\be
A*A*A*A\cdots* A=A^{*k}
\ee
\end{defn}

\begin{defn}[$\mathcal P_{\gamma, {\bf k}}$ and $\mathcal P_{\gamma, k}$] \label{taz}
For $k \in \mathbb R$ and ${\bf k}=(k_1, k_2, \cdots, k_s)\in \mathbb R^s$ , $\gamma=\Phi(a, b)$, we define
\be\label{taz1}
\mathcal P_{\gamma, k}G_{ij}:=G^{*_\gamma (k+1)}_{ij}
\ee
and
\be \label{taz2}
\mathcal P_{\gamma, {\bf k}} \left( \prod_{t=1}^s G_{i_t j_t}\right) := \prod_t \left(\mathcal P_{\gamma, k_t} G_{i_t j_t}\right) = \prod_{t=1}^s G^{*_\gamma(k_t+1)}_{i_t j_t} .
\ee
If $\mathcal G_1$ and $\mathcal G_2$ are products of matrix entries of $G$'s as above, then { we define}
\be
\mathcal P_{\gamma, {\bf k}}(\mathcal G_1+\mathcal G_2) { :=} \mathcal P_{\gamma, {\bf k}} \,\mathcal G_1+\mathcal P_{\gamma, {\bf k}}\, \mathcal G_2
\ee
Note that $\mathcal P_{\gamma, {\bf k}}$ and $\mathcal P_{\gamma, { k}}$ are not linear operators but just notations we use for simplification. Similarly, for the product of the entries of the matrix $G-m_{sc} = G-m_{sc}I$, we define
\be \label{taz2wt}
\wt {\mathcal P}_{\gamma, {\bf k}} \left( \prod_{t=1}^s \left[G  -m_{sc}  \right]_{i_t j_t} \right) := \prod_t \left(\wt {\mathcal P}_{\gamma, k_t} \left[G -m_{sc}  \right]_{i_t j_t} \right),
\ee
where
\be
\wt {\mathcal P}_{\gamma, k} (G_{ij}-m_{sc})_{ij} = 
	\begin{cases}
	(G_{ij}-m_{sc})_{ij} & {\rm if} \quad i =j \quad {\rm  and} \quad k=0, \\
	\mathcal P _{\gamma, k} G_{ij} =G^{*_\gamma(k+1)}_{ij} & {\rm otherwise}.
	\end{cases}
\ee
\end{defn}

Using Definition \ref{taz2}, we may write, for example,
\be
\mathcal P_{\gamma, {\bf k}} \left( \prod_{t=1}^s G^{\gamma } _{i_t j_t}\right) = \prod_{t=1}^s  S^{*_{\gamma}(k_t+1)}_{i_t j_t}
,\quad\mathcal P_{\gamma, {\bf k}} \left( \prod_{t=1}^s G^{\gamma-1} _{i_t j_t}\right) = \prod_{t=1}^s  T^{*_{\gamma}(k_t+1)}_{i_t j_t}.
\ee
With the fact that $G^{*s} \mathcal I_\gamma G^{*t}= G^{*(s+t)}$, one can easily see that for $k\in \mathbb R$ and ${\bf k}\in \mathbb R^{k+1}$, 
\be\label{PPP}
\mathcal P_{\gamma, {\bf k}}(\mathcal P_{\gamma,k } G_{ij})=\mathcal P_{\gamma,k+|{\bf k}| } G_{ij}
\ee
Here, note that $(\mathcal P_{\gamma,k } G_{ij})$ is the sum of the products of the matrix entries of $G$, where each product contains $k+1$ matrix entries of $G$. 

Using the definitions above, we have the following lemma from the bound \eqref{813}. Recall that $R$, $S$, and $T$ depend on $\gamma$.
\begin{lem} 
For any ${\bf k\in \mathbb R^s}$, $\gamma$, $\gamma'$ and $i_1,j_1, \cdots, i_s,j_s$, we have with $3$-high probability that
\be\label{bG4}
\left|\mathcal P_{\gamma', {\bf k}} \left( \prod_{t=1}^s A_{i_t j_t}\right) \right|,\; \;
\left|\wt {\mathcal P}_{\gamma', {\bf k}} \left( \prod_{t=1}^s (A-m_{sc}) _{i_t j_t}\right) \right|\leq 4^{s+|{\bf k}|+1},
\ee
where $A$ can be $R$, $S$, or $T$.
\end{lem}

The following lemma shows how we expand the expectation of the term $S_{i_1 j_1} S_{i_2 j_2} \cdots S_{i_s j_s}$. 

\begin{lem} \label{lem:S_expansion} Let $S=(H^{\gamma}-z)^{-1}$ as above and $\Phi(a,b)=\gamma$. Assume $z \in S(C_0)$ for $C_0$ in \eqref{813}. Fix $s= O(\varphi)$ and $\zeta=O(\varphi)$. Then, for any  
$$
S_{i_1 j_1}(z) S_{i_2 j_2}(z) \cdots S_{i_s j_s}(z),
$$ 
we have with $\al=(\al_1, \al_2,\cdots, \al_s) \in \mathbb R^s$, $|\al|=\sum \al_i$,
\begin{equation} \label{eq:S_exp2}
\E  \prod_{t=1}^s S_{i_t j_t} =\sum_{0\leq k\leq 4}A_\al \E [ (-v_{ab})^{k}]+ \sum_{\al_1, \al_2, \ldots, \al_s\ge 0}^{5\le |\al|\leq 2\zeta/\phi },
\mathcal A_{\al} \,\E \,\mathcal P_{\gamma,\al}\prod_{t=1}^s S _{i_t j_t}  + O(N^{-\zeta  })
\end{equation}
where $A_i$ $(0\leq i\leq 4)$ depend only on $R$, $\mathcal A$'s are independent of $(i_t, j_t)$, $1 \leq t \leq s$, and 
\be\label{boundcalA}
\mathcal |A_{\al }| \leq  N^{-  |{\bf \al}| \phi /10}  N^{-2}.
\ee
Similarly, we have 
\begin{equation} \label{eq:S_exp3}
\E  \prod_{t=1}^s ((S-m_{sc} )_{i_t j_t} ) =\sum_{0\leq \al\leq 4}\wt A_\al \E [ (-v_{ab})^{\al_l}]+ \sum_{\al_1, \al_2, \ldots, \al_s\ge 0}^{5\le |\al|\leq 2\zeta/\phi } \mathcal A_{\al} \,\E \,\wt {\mathcal P}_{\gamma,\al}\prod_{t=1}^s S _{i_t j_t} + O(N^{-\zeta  }),
\end{equation}
where again $\wt A_i$ $(0\leq i\leq 4)$ depend only on $R$. 

Furthermore, as \eqref{eq:S_exp2}, we have 
\begin{equation} \label{eq:S_exp4}
\E  \prod_{t=1}^s S_{i_t j_t} =\E  \prod_{t=1}^s R_{i_t j_t} + \sum_{\al_1, \al_2, \ldots, \al_s\ge 0}^{1\le |\al|\leq 2\zeta/\phi }
  \wt {\mathcal A}_{\al }\, \E\,\mathcal P_{\gamma,\al}\prod_{t=1}^s S _{i_t j_t}  + O(N^{-\zeta  })
\end{equation}
where $\wt {\mathcal A}$ are independent of $(i_t, j_t)$, $1\leq t\leq s$, and 
\be\label{boundcalA2}
| \wt {\mathcal A}_{\al }| \leq  N^{-  |{\bf \al}| \phi /10}. 
\ee
Note that the terms $\mathcal A$ and $ \wt {\mathcal A}$ depend on $\gamma$. 
\end{lem}
 
We prove this lemma later in the section.
 
Recall that we let $S=G^\gamma=(H_\gamma-z)^{-1}$, $T=G^{\gamma-1}=(H_{\gamma-1}-z)^{-1}$. Note that the entires of $H_{\gamma}$ and $H_{\gamma -1}$ coincide except for the $(a, b)$ and $(b, a)$ entries, where $(H_{\gamma})_{ab} = (H_{\gamma})_{ba} =v_{ab}=h_{ab}$ and $(H_{\gamma-1})_{ab}= (H_{\gamma-1})_{ba} = w_{ab}=\wt h_{ab}$. It is obvious that a result similar to Lemma \ref{lem:S_expansion} holds for the product of $T$. Thus, as in \eqref{eq:S_exp2} we define the notation $\mathcal A^{\gamma, a}$, $a=0,1$ as follows:
\be
\E  \prod_{t=1}^s S_{i_t j_t} =\sum_{0\leq k\leq 4}\E A_\al \E [ (-v_{ab})^{k}]+ \sum_{\al_1, \al_2, \ldots, \al_s\ge 0}^{5\le |\al|\leq 2\zeta/\phi }
  \mathcal A^{\gamma, 0}_{\al } \E\mathcal P_{\gamma,\al}\prod_{t=1}^s S _{i_t j_t}   + O(N^{-\zeta  }),
\ee
\be
\E  \prod_{t=1}^s T_{i_t j_t} =\sum_{0\leq k\le  4}\E A_\al \E [ (-w_{ab})^{k}]+ \sum_{\al_1, \al_2, \ldots, \al_s\ge 0}^{5\le |\al|\leq 2\zeta/\phi }
  \mathcal A^{\gamma, 1}_{\al } \E\mathcal P_{\gamma,\al}\prod_{t=1}^s T _{i_t j_t}  + O(N^{-\zeta  }).
\ee

Using these two identities, we have 
\begin{equation} \begin{split} \label{zddd0}
&\E \prod_{t=1}^s G^{\gamma_{\rm max}} _{i_t j_t}-\E\prod_{t=1}^s G^{0} _{i_t j_t} \\
&=\sum_{\gamma}\left(\E \prod_{t=1}^s G^{\gamma} _{i_t j_t}-\E\prod_{t=1}^s G^{\gamma-1} _{i_t j_t}\right)\\ 
&=\sum_{\gamma}\sum_{{\bf k}: 5\leq |{\bf k}|\leq 2\zeta/\phi }
  \mathcal A^{\gamma,0}_{\bf k}
\,\E \ \mathcal P_{\gamma, {\bf k}} \left(\prod_{t=1}^s G^{\gamma} _{i_t j_t}\right) 
 -  \mathcal A^{\gamma,1}_{\bf k}\,\E \,\mathcal P_{\gamma, {\bf k}}  \left(\prod_{t=1}^s G^{\gamma-1} _{i_t j_t}\right)+ O(N^{-\zeta  }) 
\end{split} \end{equation}
where we used that $A_k$ ($0\leq k\leq 4$) depends only on $R$ and the first four moments of $v_{ab}=(H_\gamma)_{ab}$ and $w_{ab}=(H_{\gamma-1})_{ab}$ match. Then, we obtain that 
\be \label{zddd}
\left|\E \prod_{t=1}^s G^{\gamma_{\rm max}}_{i_t j_t}\right|\leq \left|\E\prod_{t=1}^s G^{0} _{i_t j_t}\right|   
+\sum_{\gamma} \sum_{a=0, 1} \sum_{{\bf k}: 5\leq |{\bf k}|\leq 2\zeta/\phi } | \mathcal A^{\gamma, a}_{\bf k}|
 \left|\E \mathcal P_{\gamma, {\bf k}}  \left(\prod_{t=1}^s G^{\gamma-a}_{i_t j_t}\right)\right| 
+ O(N^{-\zeta  }).
\ee
For the terms that belong to the fixed $|{\bf k}|=k$, we can see from \eqref{bG4} and \eqref{boundcalA} that they are bounded by 
\be\label{smbl}
N^{-\frac{k\phi}{10}}s^k4^{k+s}\le N^{-\frac{k\phi}{20}}4^{k+s}\leq N^{-\frac{k\phi}{21}}4^s \, .
\ee
Then, the second part of the r.h.s. of \eqref{zddd} is less than $C N^{-\frac{5\phi}{21}}4^s$ where we used $k\ge 5$.

Recall that $P_{\gamma, {\bf k}}  \left(\prod_{t=1}^s G^{\gamma-a} _{i_t j_t}\right)$ is also a sum of the products of $G$. Using the result \eqref{zddd0} again on the $\left|\E \mathcal P_{\gamma, {\bf k}}  \left(\prod_{t=1}^s G^{\gamma-a} _{i_t j_t}\right)\right|$, where we replace the $\gamma_{\max}$ with $\gamma-a$ in the left hand side, we obtain the following bound as in \eqref{zddd}:
\begin{equation} \begin{split}
&\left|\E \mathcal P_{\gamma, {\bf k}}  \left(\prod_{t=1}^s G^{\gamma-a} _{i_t j_t}\right)\right| \\
&\leq \left|\E \mathcal P_{\gamma, {\bf k}}  \left(\prod_{t=1}^s G^{0} _{i_t j_t}\right)\right|+
\sum_{\gamma'}
\sum_{a'=0, 1}
\sum_{{\bf k'}: 5\leq |{\bf k'}|\leq 2\zeta/\phi }  |\mathcal A^{\gamma', a'}_{\bf k'}|
 \left|\E  \mathcal P_{\gamma', {\bf k'}} \mathcal P_{\gamma, {\bf k}}  \left(\prod_{t=1}^s G^{\gamma'-a'} _{i_t j_t}\right)\right|+ O(N^{-\zeta  }),
\end{split} \end{equation}
where ${\bf k}'\in \mathbb R^{s+|{\bf k}|}$. Thus, together with  \eqref{zddd}, we have 
\begin{equation} \begin{split} \label{722s}
\left|\E \prod_{t=1}^s G^{\gamma_{\rm max}} _{i_t j_t}\right|
\leq &\left|\E\prod_{t=1}^s G^{0} _{i_t j_t}\right|  
+ 
\sum_{{\bf k}: 5\leq |{\bf k}|\leq 2\zeta/\phi }  \sum_{\gamma, a}|\mathcal A^{\gamma, a}_{\bf k}|
\left|\E \mathcal P_{\gamma, {\bf k}}  \left(\prod_{t=1}^s G^{0} _{i_t j_t}\right)\right|
\\
& \quad + \sum_{\gamma', \gamma}
\sum_{a, a' }
\sum_{{\bf k'},  {\bf k}} 
|\mathcal A^{\gamma, a}_{\bf k } \mathcal A^{\gamma', a'}_{\bf k'}|
 \left|\E  \mathcal P_{\gamma', {\bf k'}} \mathcal P_{\gamma, {\bf k}}  \left(\prod_{t=1}^s G^{\gamma'-a'} _{i_t j_t}\right)\right|+ O(N^{-\zeta  }).
\end{split} \end{equation} 

Since $|\mathcal A_{\bf k}| \leq (N^{-2-\phi |{\bf k}|/10})$, for the terms, in the second line of \eqref{722s}, that belong to the fixed $|{\bf k}|+|{\bf k'}|=k$, as in \eqref{smbl}, it is easily to be bounded by 
$$
N^{-\frac{k\phi}{20}}4^{k+s}\leq N^{-\frac{k\phi}{21}}4^s.
$$
Hence, the sum in the second line of \eqref{722s} is less than $N^{-\frac{10\phi}{21}}4^s + O(N^{-\zeta  })$, where we used that $|{\bf k}|+|{\bf k'}|\ge 10$. Repeating this process, we make the sum smaller and smaller. At the end, we obtain that 
\begin{equation} \begin{split} \label{HHtt}
&\left|\E \prod_{t=1}^s G^{\gamma_{\rm max}}_{i_t j_t}\right| \\
&\leq \sum_{n=0}^{6\zeta/\phi}\sum_{\gamma_1, \gamma_2,\cdots, \gamma_n}
\sum_{a_1, a_2,\cdots, a_n}
\sum_{{ \bf k_1},  {\bf k_2}\cdots, {\bf k_n}} 
|\prod_j\mathcal A^{\gamma_j, a_j}_{{\bf k}_j}|
 \left|\E  \mathcal P_{\gamma_n, {\bf k}_n}\cdots \mathcal P_{\gamma_1, {\bf k}_1}  \left(\prod_{t=1}^s G^{0} _{i_t j_t}\right)\right|+O(N^{-\zeta}),
\end{split} \end{equation}
where 
\be\label{conK}
n\leq 6\zeta/\phi,\quad {\bf k_1}\in \mathbb R^s,\quad {\bf k_2}\in \mathbb R^{s+|{\bf k_1}|}, \quad {\bf k_3}\in \mathbb R^{s+|{\bf k_1}| +|{\bf k_2}|} , \quad{\rm  etc., \quad and  }\quad 5\leq |{\bf k}_i|\leq 2\zeta/\phi. 
\ee
Using the bound $|\mathcal A_{\bf k}| \leq (N^{-2-\phi |{\bf k}|/10})$ again with $s, \zeta \leq O(\varphi)$,  we have 
\begin{equation} \begin{split} \label{HHtt2}
&\left|\E \prod_{t=1}^s G^{\gamma_{\rm max}} _{i_t j_t}\right| \\
&\leq \left|\E \prod_{t=1}^s G^{0} _{i_t j_t}\right|
+\max_{ {\bf k}, n}( N^{-2})^{n}(N^{-\phi/20})^{\sum_i |{\bf k_i}|}\sum_{\gamma_1, \gamma_2,\cdots, \gamma_n}\left|\E  \mathcal P_{\gamma_n, {\bf k}_n}\cdots \mathcal P_{\gamma_1, {\bf k}_1}  \left(\prod_{t=1}^s G^{0} _{i_t j_t}\right)\right|+O(N^{-\zeta}).
\end{split} \end{equation}
Note that the first term in the right hand side is from the sum with $n=0$ in \eqref{HHtt} and the $\gamma_{\max}$ in the left hand side can be replaced with any $0\leq \gamma_0 \le \gamma_{\max}$.

Since these $\mathcal A$ and $\mathcal P$ are independent of $i_t$ and $j_t$ $(1 \leq t \leq s)$, we may consider a linear combination of \eqref{HHtt2}, i.e.,
for a coefficient function $f(I,J)$ with
\be
\sum_{I, J} f(I,J)=1, {  \quad f(I, J) \geq 0, }\quad I=(i_1, i_2, \cdots, i_s), \quad  J=(j_1, j_2, \cdots, j_s),
\ee
we have 
\begin{align}\label{HHtt3}
\left|\E\sum_{I,J} f(I,J)\prod_{t=1}^s G^{\gamma_{\rm max}} _{i_t j_t}\right|
\leq& \left|\E \sum_{I,J} f(I,J)\prod_{t=1}^s G^{0} _{i_t j_t}\right|\\\nonumber
+&\max_{ {\bf k}, n, \gamma } (N^{-\phi/20})^{\sum_i |{\bf k_i}|} 
\left|\E \sum_{I,J} f(I,J) \mathcal P_{\gamma_n, {\bf k}_n}\cdots \mathcal P_{\gamma_1, {\bf k}_1}  \left(\prod_{t=1}^s G^{0} _{i_t j_t}\right)\right|+O(N^{-\zeta}).
\end{align}
In application, we let $H=H^{\gamma_{\max}}$, $\wt H= H^0$. 

Similarly, with \eqref{eq:S_exp4} we can extend \eqref{HHtt3} to 
\begin{equation} \begin{split} \label{bssm}
& \left|\E \sum_{I,J} f(I,J)\prod_{t=1}^s ((G -m_{sc}) _{i_t j_t})\right| \\
&\leq
\left|\E \sum_{I,J} f(I,J)\prod_{t=1}^s ((\wt G -m_{sc}) _{i_t j_t})\right| \\
& \quad +
 \max_{\gamma, n, {\bf k}}(N^{-\phi/20})^{\sum_i |{\bf k_i}|}
 \left|\E \sum_{I,J} f(I,J)\wt {\mathcal P}_{\gamma_n, {\bf k}_n}\cdots  \wt {\mathcal P}_{\gamma_1, {\bf k}_1}  \left(\prod_{t=1}^s (\wt G -m_{sc})  _{i_t j_t}\right)\right|+O(N^{-\zeta}).
\end{split} \end{equation}
We note that, in the case $f=N^{{  -s}}\prod \delta_{i_tj_t}$, the left hand side equals to $\E (m-m_{sc})^s$ and the first term in the right hand side equals to $\E (\wt m-m_{sc})^s$.

Now, we first use \eqref{HHtt} to prove Lemma \ref{lem: WBGij}.

\begin{proof}[Proof of Lemma \ref{lem: WBGij}]
Let $H=H^{\gamma_{\max}}$ and $\wt H= H^0$. Since $\wt H$ satisfies the bounded support condition with $q\sim N^{-1/2}/\log N$, we have from \eqref{Gij estimate} that, for $z$ satisfying the assumption in Lemma \ref{lem: WBGij}, with high probability,     
\be\label{tddc}
|\wt G_{ij}|\leq \varphi^{O(1)}\left(\sqrt{\frac{\im m_{sc}}{N\eta}}+\frac{1}{N\eta}\right)+ 2\delta_{ij}.
\ee
On the other hand, we have a trivial bound $ |\wt G_{ij}|\leq N$. (See \eqref{813tr}.) We now apply \eqref{HHtt} on $G_{ij} \overline{G_{ij}}$  with $s=2$ and $\zeta=1$. To prove the lemma, it suffices to show that the following holds for any ${\bf k }_1$, ${\bf k }_2\cdots, {\bf k }_n$, for $n$ satisfying \eqref{conK}:
\be
N^{-2n} \sum_{\gamma_1, \gamma_2,\cdots, \gamma_n}\left|\E  \mathcal P_{\gamma_n, {\bf k}_n}\cdots \mathcal P_{\gamma_1, {\bf k}_1} \wt G_{ij} \overline{\wt G_{ij}} \right| \leq \varphi^C\left( \frac{\im m_{sc}}{N\eta} +\frac{1}{(N\eta)^2} \right).
\ee
With Lemma \ref{lm:msc}, it is easy to check that the right hand side is larger than $N^{-1}$. Let $\Phi(\gamma_t)=(a_t, b_t)$. It only remains to prove that 
\be\label{sysz}
\max_{\gamma_1, \gamma_2,\cdots, \gamma_n: \;i,j\notin \cup_{1\le t\le n}\{a_t, b_t\}}\left|\E  \mathcal P_{\gamma_n, {\bf k}_n}\cdots \mathcal P_{\gamma_1, {\bf k}_1} \wt G_{ij} \overline{\wt G_{ij}} \right|\leq \varphi^C \left( \frac{\im m_{sc}}{N\eta} +\frac{1}{(N\eta)^2}\right).
\ee
By definition, $\mathcal P_{\gamma_n, {\bf k}_n}\cdots \mathcal P_{\gamma_1, {\bf k}_1} \wt G_{ij} \overline{\wt G_{ij}}$ is a finite sum of the products of the matrix entries of $G$ and $\overline{G}$. Furthermore, for each product, there exist at least two off diagonal terms, since the index $i$ {appears exactly twice} and there is no $G_{ii}$ term in $ \mathcal P_{\gamma_n, {\bf k}_n}\cdots \mathcal P_{\gamma_1, {\bf k}_1} \wt G_{ij} \overline{\wt G_{ij}}$. From the existence of these two off-diagonal terms and from \eqref{tddc}, we obtain \eqref{sysz} and complete the proof of Lemma \ref{lem: WBGij}.
\end{proof}

Next, we use \eqref{bssm} to prove Lemma \ref{lem:p moment}.
\begin{proof}[Proof of Lemma \ref{lem:p moment}]
For simplicity, we prove instead that 
\be
\left| \E (m-m_{sc})^p \right| \leq \left| \E (\wt m-m_{sc})^p \right| + (Cp)^{Cp} (X^2+Y+\varphi^CN^{-1})^p.
\ee
(The proof of \eqref{sqbm} is exactly the same except that it involves more terms with more complicated expressions.) Using \eqref{bssm} with $i_t=j_t$ and $s=\zeta=p$ and $f(I,J)=N^{-p}\prod_t \delta_{i_t j_t}$, since $\mathcal A$ are independent of $i_t, j_t$ for any $1 \leq t \leq s$, we have 
\begin{align}\label{bssm2}
\left|\E   ( m-m_{sc})^p \right|
\leq  &
\left|\E ( \wt m-m_{sc})^p \right|
\\\nonumber
+&
 \max_{\gamma, {\bf k},n}(N^{-\phi/20})^{\sum_i |{\bf k_i}|}
 \left|\E \frac{1}{N^{p}}\sum_{i_1 \cdots, i_p} \wt {\mathcal P}_{\gamma_n, {\bf k}_n}\cdots  \wt {\mathcal P}_{\gamma_1, {\bf k}_1}  \left(\prod_{t=1}^p (\wt G_{i_t,i_t} -m_{sc}) \right)\right|+O(N^{-\zeta}).
\end{align}
With the assumption \eqref{asXY}, the first term in the right hand side is bounded with $(Y+\varphi^C N^{-1})^p$, where for the bad event of probability space we used \eqref{813tr}. In order to complete the proof, we now only need to bound the second term in the the right hand side of \eqref{bssm2}. For any fixed $\gamma_1, \cdots, \gamma_n$, ${\bf k_1}, \cdots, {\bf k_n}$ and $i_1, \cdots,  i_p$ satisfying \eqref{conK}, we know that  
\be\label{dsb}
\wt {\mathcal P}_{\gamma_n, {\bf k}_n}\cdots  \wt {\mathcal P}_{\gamma_1, {\bf k}_1}  \left(\prod_{t=1}^p (\wt G_{i_t,i_t} -m_{sc}) \right)
\ee
is the sum of at most $C^{\sum |{\bf k}_i|}$ products of $\wt G_{ij}$ (including the terms with $i = j$) and $(\wt G_{ii}-m_{sc})$, where the total number of $\wt G_{ij}$ and $(\wt G_{ii}-m_{sc})$ is $\sum |{\bf k}_i|+p {\,=\,} O(\varphi^2)$.  Since $\wt G$ has a rough bound $|\wt G(z)|\le \eta^{-1}\leq N$, we know \eqref{dsb} is always less than $N^{O(\varphi^2)}$.  With the assumption that \eqref{asXY} holds with 3-high probability, we noticed that  the event that \eqref{asXY} does not hold is negligible. Futhermore, for each product in \eqref{dsb} and any fixed $t$,   $1\le t\le p$, we know there are two $i_t$'s in the indices of $G$. These two  $i_t$'s can only appear as (a) $\wt G_{i_t i_t}-m_{sc}$ in the product, or (b) $G_{i_t, a} G_{b, i_t }$, where the indices $a$ and $b$ come from some $\gamma_k$ and $\gamma_l$, $(1\leq k,l\leq n)$ via    $\mathcal P$. Thus, after averaging over $1 \leq i_t\leq N$, this term becomes (a) $\wt m-m_{sc}$, which is bounded by $Y$ (see \eqref{asXY}), or (b)
$ 
 N^{-1} \sum_{i_t} G_{i_t, a } G_{b, i_t} ,
$ 
which is bounded by $X^2+C N^{-1}$ with \eqref{asXY}. In the case (b), we also used the fact that the number of non-generic terms with $i_t = a$ or $i_t = b$ is smaller than that of generic terms by a factor $N^{-1}$, hence we bound the contribution from the non-generic terms by $C N^{-1}$.

Therefore, we have showed that after averaging $1\leq i_1, i_2, \cdots, i_p\leq N$, i.e., applying $N^{-p} \sum_{i_1 \cdots, i_p}$, each $i_t$ either contributes a factor $\wt m-m_{sc}$, i.e. $Y$, or $N^{-1} \sum_{i_t} G_{i_t a } G_{b i_t}$, i.e., $(X^2+CN^{-1})$. For any other $G$'s in the product with no $i_{t}$ ($1\le t\le p$), we simply bound them as $C$, then for any fixed $\gamma_1,\cdots, \gamma_n$, ${\bf k }_1,\cdots,{\bf k }_n$, we had proved that 
  \be\left|\E \frac{1}{N^{p}}\sum_{i_1 \cdots, i_p} \wt {\mathcal P}_{\gamma_n, {\bf k}_n}\cdots  \wt {\mathcal P}_{\gamma_1, {\bf k}_1}  \left(\prod_{t=1}^p (\wt G_{i_t,i_t} -m_{sc}) \right)\right|\leq C^{\sum_i |{\bf k_i}|+p}(X^2+Y+CN^{-1})^p .
 \ee
With \eqref{bssm2}, this completes the proof of Lemma \ref{lem:p moment}.   
\end{proof} 

\begin{proof}[Proof of Lemma \ref{lem:S_expansion}]
Choose 
\be\label{defxin}
\xi=2\zeta/\phi. 
\ee

We apply the expansion 
\begin{equation}
S = R - RVR + (RV)^2 R - \cdots + (-1)^{\xi} (RV)^{\xi} R + (-1)^{\xi + 1} (RV)^{\xi + 1} S .
\end{equation}
With the condition $|v_{ab}|\leq N^{-\phi}$, we note that the last term in this expansion is $O(N^{- \zeta})$. Thus,
\begin{equation}\label{cfdrx}
S = R - RVR + (RV)^2 R - \cdots + (-1)^{\xi} (RV)^{\xi} R + O(N^{-\xi\phi }).
\end{equation}
Since $v_{kl} = 0$ if $\{ k, l \} \neq \{ a, b \}$, we have
\begin{equation}
([-RV]^m R)_{i_t j_t} = \sum_{ (a_i, b_i) \in \{ (a, b), (b,a) \};\, 1 \leq i \leq m} R_{i_t a_1} V_{a_1 b_1} R_{b_1 a_2} \cdots V_{a_m b_m} R_{b_m j_t} \,.
\end{equation}
We note that  $V_{ab}=V_{ba}=v_{ab}=h_{ab}$. Using Definition \ref{def*} and Definition \ref{taz}, we have
\be\label{wbyj}
 [-RV]^{k } R=R^ {*(k+1)}(-h_{ab})^k
 , \quad S=\sum_{0\le k\leq  \xi} ( R^ {*(k+1)})(-h_{ab})^k + O(N^{- \xi\phi}).
\ee
Similarly, 
\be
R_{ij}=\sum_{0\le k\leq  \xi} (\mathcal P_{\gamma, k}S_{ij})(h_{ab})^k + O(N^{- \xi\phi}), \quad (R -m_{sc}) _{ij} =\sum_{0\le k\leq  \xi} (\wt {\mathcal P}_{\gamma, k}S_{ij}) (h_{ab})^k + O(N^{- \xi\phi}).
\ee
For this reason, we only show the proof of \eqref{eq:S_exp2}, and \eqref{eq:S_exp3} can be proved analogously. The proof of \eqref{eq:S_exp4} will roughly be explained at the end of this proof. 

Using \eqref{wbyj} and \eqref{813}, with 3-high probabilily, we have that (with definition \ref{taz})
\be\label{dmty}
\prod_{t=1}^{s} S_{i_t j_t} 
= \sum _{0\le k\leq  \xi s;} \sum_{{\bf k}\in    I^s_{ \xi, k}}
\left(\mathcal P_{\gamma, {\bf k}}\prod_{t=1}^s R_{i_tj_t} \right) (-h_{ab})^{k_t}+ O(2^sN^{-\xi\phi})
\ee
where 
\be
{\bf k}:=(k_1,k_2, \cdots, k_s), 
 \quad I^a_{b,c}:=\{{\bf k}\in \mathbb R^a: 0\le k_i\leq b,   \sum k_i=c\}
\ee
Note that, from the above definition,
\be\label{boundI}
|I^a_{b,c}|\leq a^c \, .
\ee
We note that the term in \eqref{dmty} belonging to $k=0$ is $\prod_{t=1}^{s} R_{i_t j_t}$. For the terms belonging to $k\ge \xi$, using \eqref{bG4}, we know that with 3-high probability they are bounded by
\be
\sum_{k\ge \xi}s^k 4^{k+s}N^{-k\phi} \le O( s^{\xi} 4^{\xi+s} N^{-\xi\phi}),
\ee
where $s^k$ comes from $\sum_{{\bf k}\in I^s_{\xi, k}}$. Hence, 
\be\label{hxxt}
\prod_{t=1}^{s} S_{i_t j_t} 
=\prod_{t=1}^{s} R_{i_t j_t}+  \sum _{1\leq k\leq \xi } \left( \sum_{{\bf k} \in I^s_{\xi, k}}
\mathcal P_{\gamma, {\bf k}} \prod_{t=1}^s R _{i_t j_t}  \right) (-h_{ab})^k+ O(  s^{\xi} 4^{\xi+s} N^{-\xi\phi}).
\ee

Recall that $(S^{*s})_{ij}$ is a sum of terms as in Definition \ref{def*}. As a special case, consider a term in $(S^{*s})_{ij}=\mathcal P_{\gamma, s-1}S_{ij}$ and rewrite it as $S_{i_1 j_1} S_{i_2 j_2} \cdots S_{i_s j_s}$. We have 
\be 
(S^{*s})_{ij} = (R^{*s})_{ij}+\sum _{1\leq k\leq \xi } \left( \sum_{{\bf k}\in I^s_{\xi, k}}
(\mathcal P_{\gamma, {\bf k}}(R^{*s})_{ij}) \right)(-h_{ab})^k
+ O( s^{  \xi} 4^{ \xi+s} N^{-\xi\phi}).
\ee
Then, with \eqref{PPP}, we have (here we replaced $s$ with $s+1$ for simplicity)
\be \label{cyyl}
 \mathcal P_{\gamma, s }S_{ij}=\mathcal P_{\gamma, s }R_{ij}+ \sum _{1\leq k\leq \xi }| I^{s+1}_{\xi, k}|  (\mathcal P_{\gamma, {s+k }}R)_{ij} (-h_{ab})^k+ O(  s^{\xi} 4^{ \xi+s} N^{-\xi\phi}) , 
\ee
Define, for $0\leq k\leq 4$, 
\be
A_k:=\sum_{{\bf k}\in I^s_{\xi, k}} \mathcal P_{\gamma, {\bf k}}R_{i_t j_t} = \sum_{{\bf k}\in I^s_{\xi, k}} \prod_{t=1}^s R^{*(k_t+1)}_{i_t j_t} 
\ee
Clearly, they depend only on $R$. Thus, as in \eqref{hxxt}, with 3-high probability, 
\be
\prod_{t=1}^{s} S_{i_t j_t} 
=\sum_{k=0}^{4}A_k(-h_{ab})^k+ \sum _{5\leq k\leq \xi }\left( \sum_{{\bf k}\in  I^s_{\xi, k}}
\mathcal P_{\gamma, {\bf k}}\prod_{t=1}^s R _{i_t j_t}  \right)(-h_{ab})^k+
O(  s^{\xi} 4^{ \xi+s} N^{-\xi\phi}).
\ee
We take the expectation $\E$ in the both sides of the equation. Recall that the good event holds with 3-high probability and the entries of $S$ and $R$ are bounded by $\eta^{-1} \leq N$. (see \eqref{813tr}). Furthermore, in this proof, no products have more than $O(\varphi^2)$ entries of $S$ or $R$'s. Thus, when taking the expectation $\E$, we can simply ignore the set of bad event.

To simplify the notation, we define
\be
n_k:= \E (-h_{ab})^k.
\ee
Then, we get
\begin{equation} \label{eq:S expand2}
\E\prod_{t=1}^{s} S_{i_t j_t} 
=\sum_{k=0}^{4} n_k\E A_k+ \sum _{5\leq k\leq \xi } n_k\left( \sum_{{\bf k}\in    I^s_{\xi, k}}
\E\mathcal P_{\gamma, {\bf k}}\prod_{t=1} ^sR _{i_tj_t}   \right) +
O(  s^{  \xi} 4^{ \xi+s} N^{-\xi\phi})
\end{equation}
To estimate $\E\mathcal P_{\gamma, {\bf k}}\prod_{t=1} ^s R_{i_t j_t} 
=\E \prod_{t} \mathcal P_{\gamma, k_t} R _{i_t j_t}$, we first use \eqref{cyyl} and obtain 
\be
\mathcal P_{\gamma, k_t} S_{i_tj_t}=\mathcal P_{\gamma, k_t} R _{i_tj_t}+\sum _{1\le l_t\le \xi -k }
|I^{k_t+1}_{\xi-k, l_t}|  \mathcal P_{\gamma, (k_t+l_t)} R _{i_tj_t}  (-h_{ab})^{l_t}+ O( (4k_t+4)^{\xi-k} 4^{k_t+1}  N^{-\xi\phi+k\phi}).
\ee
Note that from \eqref{boundI} we have a bound 
$$
|I^{k_t+1}_{\xi-k, l_t}|  \leq (k_t+1)^{l_t}.
$$ 
From that $|\mathcal P_{\gamma, k_t} S_{i_tj_t}|\leq 4^{k_t+{ 2}}$ (see \eqref{bG4}), for $0\le k_1,k_2, \cdots, k_s\le \zeta  /\phi $, $\sum k_i=k$ and $k\leq \zeta/\phi$, we obtain that
\begin{equation} \begin{split} \label{dydy}
\mathcal P_{\gamma, {\bf k}}\prod_{t=1} ^sS _{i_tj_t} =&\mathcal P_{\gamma, {\bf k}}\prod_{t=1} ^sR _{i_tj_t} +\sum _{1\le l\le( \xi-k)s  }
\left(\sum_{{\bf l}\in I^s_{ \xi-k, l}}
\prod_t|I^{k_t+1}_{ \xi-k, l_t}|  \mathcal P_{\gamma, k_t+l_t} R _{i_tj_t}\right) (-h_{ab})^l \\
& \quad+ O(  (4k+4)^{\xi-k} 4^{k+s}  N^{-\xi\phi+k\phi}).
\end{split} \end{equation}
For the terms that belong to $l \ge \xi-k$, from \eqref{bG4} and \eqref{boundI}, they are bounded by 
$$ 
\sum_{l\ge \xi-k}s^l(k+1)^l4^{k+l+s} N^{-l\phi} \leq  (sk+s)^{ \xi-k} 4^{\xi+s} N^{-\xi\phi+k\phi}.
$$
Then, the upper bound of $l$ in \eqref{dydy} can be reduced to $\xi-k$ as follows: 
\begin{equation} \begin{split} \label{gh}
\mathcal P_{\gamma, {\bf k}}\prod_{t=1} ^sS _{i_tj_t} =&\mathcal P_{\gamma, {\bf k}}\prod_{t=1} ^sR _{i_tj_t} +\sum _{1\le l\le  \xi-k  }
\left(\sum_{{\bf l}\in I^s_{ \xi-k, l}}
\prod_t|I^{k_t+1}_{ \xi-k, l_t}|  \mathcal P_{\gamma, k_t+l_t} R _{i_tj_t}\right) (-h_{ab})^l \\
& \quad + O(  (sk+s)^{\xi-k} 4^{{ \xi} +s}  N^{-\xi\phi+k\phi}) 
\end{split} \end{equation}

We observe that
 \be
 \prod_t|I^{k_t+1}_{ \xi-k, l_t}|  \mathcal P_{\gamma, k_t+l_t} R _{i_tj_t}
 = \left(\prod_t|I^{k_t+1}_{ \xi-k, l_t}| \right) \left(\mathcal P_{\gamma, {\bf k+l}}\prod_t R _{i_tj_t}\right), \quad {\bf l}=(l_1,l_2, \cdots, l_s).
 \ee
Inserting it into \eqref{eq:S expand2}, with $|h_{ab}|\leq N^{-\phi}$ and $|I^s_{\xi,k}| \leq s^k$,  $k\leq \xi$, we get
\begin{align}\label{off}
\E\prod_{t=1}^{s} S_{i_t j_t} 
=&\sum_{k=0}^{4}  n_k\E A_k
+
 \sum _{5\le k\le \xi  }n_k
\left( \sum_{{\bf k}\in I^s_{\xi,k}}
\E \mathcal P_{\gamma, {\bf k}}\prod_{t=1}^s S_{i_t j_t} \right)
\\\nonumber
&-\sum _{k= 5}^{ \xi }
\sum _{l=1}^{ \xi -k  }
n_kn_l\left(
\sum_{{\bf k}\in I^s_{\xi,k}}
 \sum_{{\bf l}\in I^s_{\xi-k,l}}
\left(\prod_t|I^{k_t+1}_{ \xi-k, l_t}| \right) \E\left(\mathcal P_{\gamma, {\bf k+l}}\prod_t R _{i_tj_t}\right) \right)
\\\nonumber
&+ O(\xi ( s\xi+s)^{\xi} 4^{s+\xi}   N^{-\xi\phi }) ,
\end{align}
where the single $\xi$ factor in the last term comes from $\sum_k$. For $k\ge 5$, define 
\be
A_k=  \sum_{{\bf k}\in I^s_{\xi , k}}  \mathcal P_{\gamma, {\bf k}}\prod_{t=1} ^sS _{i_t j_t} \, .
\ee
Clearly, it has at most $s^k$ terms having the form of $\mathcal P \prod S $. Applying \eqref{gh} again on $\E\mathcal P_{\gamma, {\bf k+l}}\prod_t R _{i_tj_t}$, as in \eqref{gh}, we have 
\begin{align}
 &\mathcal P_{\gamma, {\bf k+l}}\prod_t R_{i_t j_t} \nonumber \\
=&\mathcal P_{\gamma, {\bf k+l}}\prod_t S_{i_t j_t}-\sum _{1\leq o\le {  \xi-k-l} }
\left(\sum_{{\bf o}\in I^s_{\xi-k-l, o}}
\left(\prod_t|I^{k_t+l_t+1}_{ \xi-k-l, o_t}| \right) \mathcal P_{\gamma, {\bf k+l+o}}\prod_t R _{i_tj_t}\right) (-h_{ab})^o
\\\nonumber
&+ O(  (s(k+l)+s)^{\xi-k-l} 4^{{ \xi}+l+s}  N^{(-\xi+k+l)\phi}) 
\end{align}
We now insert it back to \eqref{off}, replacing the notation $k$, $l$, $o$ with $k_1$, $k_2$, $k_3$ and $k_t$, $l_t$, $o_t$ with ${ k}_{1}(t)$,  ${ k}_{2}(t)$, ${ k}_{3}(t)$, respectively. Using 
\be
|n_kn_l |
\sum_{{\bf k}\in I^s_{\xi,k}}
 \sum_{{\bf l}\in I^s_{\xi-k,l}}
 \prod_{t=1} ^s |I^{k_t+1}_ {\xi-k, l_t}| \leq N^{-(k+l)\phi}s^k s^l(k+1)^l,
\ee
we obtain
\begin{align}\label{off2}
&\E\prod_{t=1}^{s} S_{i_t j_t} \nonumber \\ 
=&\sum_{k_1=0}^{\xi}  n_{k_1}\E A_{k_1} 
 -\sum _{k_1= 5}^{ \xi  }
\sum _{k_2=1}^{ \xi-k_1  }
 n_{k_1}n_{k_2}\left(
\sum_{{\bf k_1}\in I^s_{\xi , k_1}}
 \sum_{{\bf k_2}\in I^s_{\xi -k_1, k_2}}
\left(  \prod_{t=1} ^s 
 |I^{{ k}_{1}(t)+1}_{\xi-k_1, {  k}_{2}(t)}|\right)
   \E \mathcal P_{\gamma, {\bf k_1+k_2}}\prod_t S _{i_tj_t}  \right) \nonumber
\\
&+\sum _{k_1= 5}^{ \xi  }
\sum _{k_2=1}^{ \xi-k_1  }\sum _{k_3=1}^{ \xi-k_1-k_2  }
 n_{k_1}n_{k_2}n_{k_3}\times
 \\
 & \quad \left(
\sum_{{\bf k_1}\in I^s_{\xi , k_1}}
 \sum_{{\bf k_2}\in I^s_{\xi -k_1, k_2}}
  \sum_{{\bf k_3}\in I^s_{\xi -k_1-k_2, k_3}}
\left(\prod_{t=1} ^s 
 |I^{{ k}_{1}(t)+1}_{\xi-k_1, { k}_{2}(t)}| |I^{{ k}_{1}(t)+{ k}_{2}(t)+1}_{\xi-k_1-k_2, { k}_{3}(t)}|\right)
 \E \mathcal P_{\gamma, {\bf k_1+k_2+k_3}}\prod_t R _{i_tj_t}  \right) \nonumber
\\ 
&+O(  \xi ^2 ( s\xi+s)^{\xi} 4^{s+\xi}   N^{-\xi\phi}), \nonumber
\end{align}
where the factor $\xi^2$ comes from $\sum_{k_1}\sum_{k_2}$. Define 
\be
A_{k_1,k_2}:=-\sum_{{\bf k_1}\in I^s_{\xi , k_1}}
 \sum_{{\bf k_2}\in I^s_{\xi -k_1, k_2}}
 \left(  \prod_{t=1} ^s 
 |I^{{ k}_{1}(t)+1}_{\xi-k_1, { k}_{2}(t)}|\right)
   \E \mathcal P_{\gamma, {\bf k_1+k_2}}\prod_t S _{i_t j_t}.
\ee 
Clearly, letting $k:=k_1+k_2$, we find that $A_{k_1,k_2}$ has at most  
$$
s^{k_1+k_2} \prod_{t=1} ^s 
 |I^{{ k}_{1}(t)+1}_{\xi-k_1, { k}_{2}(t)}|\leq s^{k_1+k_2}(k_1+1)^{k_2} \leq s ^{k}(k+1)^k\leq C(sk)^{ k}
$$
terms of the form $\mathcal P\prod S$. 

We repeat the previous procedure $\xi$ times. Recall that $k_i:= \sum_t k_{i }(t)$. Let
\be
\wt k_{i}:=k_{1 }+k_{2 }+\cdots+ k_{i  }, \quad \wt k_{i }(t):=k_{1 }(t)+k_{2 }(t)+\cdots+ k_{i }(t).
\ee
Define
\be \begin{split} \label{wbtnm}
& A_{k_1,k_2,k_3\cdots, k_n} \\
&=(-1)^{n-1}\sum_{{\bf k_1}\in I^s_{\xi , k_1}}
 \sum_{{\bf k_2}\in I^s_{\xi -\wt k_1, k_2}}
  \sum_{{\bf k_2}\in I^s_{\xi -\wt k_2, k_3}}
  \cdots
   \sum_{{\bf k_n}\in I^s_{\xi -\wt k_{n-1}, k_n}}
 \prod_{t=1} ^s  \left(\prod_{m=1}^{n-1}
 |I^{\wt k_{m}(t)+1}_{\xi-\wt k_m, k_{{m+1},t}}|\right)
  \mathcal P_{\gamma, (\sum_{i=1}^n {\bf k_i})}\prod_t S _{i_t j_t}.  
\end{split} \ee
Let $\# A_{k_1,k_2,k_3\cdots, k_n}$ be the number of the terms of the form $\mathcal P\prod S$ in $A_{k_1,k_2,k_3\cdots, k_n}$. Clearly, with $k:=\sum_{i}k_i$,   
\be\label{dngs}
\# A_{k_1,k_2,k_3\cdots, k_n}\leq s^{k}(k +1)^{k }  \leq C ( sk)^{ k}.
\ee
Thus, we obtain that 
\begin{equation} \begin{split} \label{off3}
\E\prod_{t=1}^{s} S_{i_t j_t} 
=&\sum_{k_1=0}^{\xi}  n_{k_1}\E A_{k_1}+\sum _{k_1= 5}^{ \xi  }
\sum _{k_2=1}^{ \xi-k_1  }
 n_{k_1}n_{k_2}A_{k_1,k_2}
+\sum _{k_1= 5}^{ \xi  }
\sum _{k_2=1}^{ \xi-k_1  }\sum _{k_3=1}^{ \xi-k_1-k_2  }
 n_{k_1}n_{k_2}n_{k_3} A_{k_1,k_2,k_3}
 \\
 &+\cdots
 \\
 &+\sum _{k_1, k_2,\cdots, k_\xi;} 
\left( \prod _{i=1}^{\xi} n_{k_i} \right)A_{k_1,k_2,\cdots, k_\xi}+O(  \xi ^\xi ( s\xi+s)^{\xi} 4^{s+\xi}   N^{-\xi\phi }) ,
\end{split} \end{equation}
where we sum up $k_1, k_2,\cdots, k_\xi$ under the condition $ k_1\ge 5, k_i\ge 1, 1\le \sum k_i\le \xi$. The factor 
$\xi^\xi$ comes from $\sum _{k_1, k_2,\cdots, k_\xi;}$. The equation \eqref{off3} implies \eqref{eq:S_exp2}. 

Now we are ready to prove \eqref{boundcalA}. Since $\al$ plays the role of $(\sum_{i=1}^n {\bf k_i})$ in \eqref{wbtnm}, we have that
\be
|\al|= \sum_t \al_t =\sum_{i}k_i
\ee
Then, we obtain that 
$$
|\mathcal A_{\al }|\leq  \sum_{\sum k_i=\sum \al_i}\#A_{k_1,k_2,\cdots, k_n}\prod_t \E|h_{ab}|^{k_t}.
$$
Note that the Wigner matrix $H$ under the assumptions of \ref{thm: new edge un}  (i.e., $H^\V$ in Lemma \ref{lem:p moment} and \ref{lem:green}) satisfies
\begin{equation}
|\E (h_{ab})^k| \leq  (\log N)N^{-2-(k-4)\phi}, \quad k\ge 5.
\end{equation}
With \eqref{dngs} and that $k_1 \ge 5$, we obtain \eqref{boundcalA}. 

Finally, we briefly explain the proof of \eqref{eq:S_exp4} and \eqref{boundcalA2}. It is almost the same as the one for \eqref{eq:S_exp2} and \eqref{boundcalA}, except changing \eqref{eq:S expand2} to
\begin{equation} 
\E\prod_{t=1}^{s} S_{i_t j_t} 
=\E\prod_{t=1}^{s} R_{i_t j_t} + \sum _{1\leq k\leq \xi } n_k\left( \sum_{{\bf k}\in I^s_{\xi, k}}
\E\mathcal P_{\gamma, {\bf k}}\prod_{t=1} ^sR _{i_tj_t}   \right) +
O(  s^{  \xi} 4^{ \xi+s} N^{-\xi\phi}),
\end{equation} 
i.e., we move the $k=1,2,3,4$ part from the first term in the right hand side to the second term. Then we keep using \eqref{hxxt} and \eqref{gh} to estimate $\E\mathcal P_{\gamma, {\bf k}}\prod_{t=1}^s R_{i_t j_t}$ as in the proof for \eqref{eq:S_exp2} and \eqref{boundcalA}.
\end{proof}

Last, we prove Lemma \ref{lem:green}.
 
\begin{proof}[Proof of Lemma \ref{lem:green}]
For simplicity, we only prove \eqref{mainpear}. The proof for \eqref{sbc22} is similar.

Recall that $\eta = N^{-2/3 - \epsilon}$. Define
\begin{equation}
x^S := \eta \; \im \tr S = \eta^2 \sum_{i, j} S_{ij} \overline{S_{ij}}, \quad x^R := \eta \; \im \tr R = \eta^2 \sum_{i, j} R_{ij} \overline{R_{ij}}.
\end{equation}
Recall also that $S=(H_\gamma-z)^{-1}$ and $R=(Q-z)^{-1}$, where all the entries of $H_{\gamma}$ and $Q$ are the same except the $(a,b)$ entries. Then, since the rank of $(H_\gamma - Q)$ is at most $2$, by interlacing theorem, we have 
\be
|\tr S-\tr R|\leq C\eta^{-1}
\ee

Together with \eqref{tby4} and \eqref{esmallfake}, with high probability,
\be\label{flld}
\quad \max_{\gamma} \{|x^S|+|x^R|\}\leq N^{C\e}.
\ee
From \eqref{Gij estimate} and \eqref{8800}, we find that
\be
\max_{\gamma}\left( |R_{ij}|+|S_{ij}|\right)\leq N^{-\phi+C\e}+C\delta_{ij}
\ee
with high probability. We also have the trivial bounds
\begin{equation}
x^S = \eta^2 \sum_{i, j} S_{ij} \overline{S_{ij}} \leq \eta^2 N^2 \eta^{-2} = N^2, \quad x^R \leq N^2, \quad 
|S|, |R|\leq \eta^{-1}\le N.
\end{equation}
Since the bad event is so small in probability space, in this proof, we basically ignore the bad set. Using the definitions we used in \eqref{eq:ordering} - \eqref{eq:green_fns}, we get a telescopic series
\begin{equation} \label{eq:telescopic_fn}
\E \: F \left( \eta^2 \sum_{i, j} G _{ij} \overline{G_{ij}}  \right) - \E \: F \left( \eta^2 \sum_{i, j} \wt G _{ij} \overline{\wt G_{ji}} \right) 
= \sum_{\gamma=1}^{\gamma_{\max}} \left[ \E \: F \left( x^{S} \right) - \E \: F \left( x^{T} \right) \right].
\end{equation}
From the Taylor expansion, we have
\begin{equation} \label{eq:taylor}
F(x^S) - F(x^R) = \sum_{s=1}^2 \frac{1}{s!} F^{(s)}(x^R) (x^S - x^R)^s  + \frac{1}{3!} F^{(3)}(\zeta_S) (x^S - x^R)^3,
\end{equation}
where $\zeta_S$ lies between $x^S$ and $x^R$, and we can obtain a similar formula for $F(x^T) - F(x^R)$ with $\zeta_T$ in place of $\zeta_S$.  

We now expand the term $S_{ij} \overline{S_{ij}}$ using \eqref{hxxt}, where the terms with the complex conjugate are treated in the same manner. Letting $\xi=3/\phi$ with $s=2$ in \eqref{hxxt}, we can see that
\be S_{ij} \overline{S_{ij}} 
=R_{ij} \overline{R_{ij}} +
 \sum _{1\leq k\leq 3/\phi}\left( \sum_{{\bf k}\in    I^2_{3/\phi, k}}
\mathcal P_{\gamma, {\bf k}}(R_{ij} \overline{R_{ij}})  \right)(-h_{ab})^k+
O(  C N^{-3})
\ee
holds with high probability. Averaging over $i, j$ and multiplying $\eta^2$ to both sides, we obtain  
\be\label{yrsh}
 x^S = x^R +
 \sum _{1\leq k\leq 3/\phi}\left( \sum_{{\bf k}\in    I^{2 }_{3/\phi, k}}
\eta^2 \sum_{i,j} \mathcal P_{\gamma, {\bf k}}\, (R_{ij} \overline{R_{ij}})   \right)(-h_{ab})^k+
O(  C N^{-3}).
\ee
 
Now, we claim that for any fixed $\bf k\neq 0$, ${\bf k}\in I^{2 }_{3/\phi, k}$, and $p=O(1)$ with $p\in 2\mathbb Z$,
\be\label{pqbh0}
\E\left| \sum _{i,j}\mathcal P_{\gamma, {\bf k}}\, (R_{i j} \overline{R_{ij}})\right |^p\leq  (N^{1+C\e})^p.
\ee
Assuming the claim \eqref{pqbh0}, with Markov inequality, we find that, for any ${\bf k}\in    I^{2 }_{3/\phi, k}$ and $\bf k\neq 0$,
\be\label{785}
| \mathcal P_{\gamma, {\bf k}}\, x^R  |:=|\eta^2\sum _{i,j} \mathcal P_{\gamma, {\bf k}}\, (R_{i j} \overline{R_{ij}})  |\leq N^{-1/3+C\e}
\ee
holds with probability with $1-N^{-D}$ for any $D>0$. 

For simplicity, we show the proof for 
\be\label{pqbh}
\left| \E\left( \sum _{i,j}\mathcal P_{\gamma, {\bf k}}\, (R_{i j} \overline{R_{ij}}) \right)^p \right| \leq  (N^{1+C\e})^p.
\ee
(The claim \eqref{pqbh0} can be proved similarly.) Using \eqref{eq:S_exp4} with $s=2p$ and  $\zeta=p$, we get
  \begin{align} 
\E  \prod_{t=1}^p \mathcal P_{\gamma, {\bf k }}\, (R_{i_t j_t}R_{j_ti_t}) 
=& \E \prod_{t=1}^p \mathcal P_{\gamma, {\bf k }}\, (S_{i_t j_t}S_{j_ti_t})\\\nonumber
 &- \sum_{\al_1, \al_2, \ldots, \al_{2p}\ge 0}^{1\le |\al|\leq 2p/\phi }
  \mathcal A_{\al }\, \E\,\mathcal P_{\gamma,\al}\left(\prod_{t=1}^p \mathcal P_{\gamma, {\bf k }}\, (S_{i_t j_t} \overline{S_{i_t j_t}})\right)  + O(N^{-p  })
\end{align}
With \eqref{boundcalA2}, in order to show \eqref{pqbh}, it only remains to prove that 
\be\label{pqbh2}
\left| \sum_{i_1,j_1\cdots, i_p,j_p}\E\,  \left(\prod_{t=1}^p \mathcal P_{\gamma, {\bf k }}\, (S_{i_t j_t} \overline{S_{i_t j_t}} )\right) \right| \leq  (N^{1+C\e})^p 
\ee
and for $\al\in \mathbb R^{2p}$, $1\le |\al|\leq 2p/\phi$, 
\be \label{pqbh3}
\left| \sum_{i_1,j_1\cdots, i_p,j_p}\E\,\mathcal P_{\gamma,\al}\left(\prod_{t=1}^p \mathcal P_{\gamma, {\bf k }}\, (S_{i_t j_t} \overline{S_{i_t j_t}} )\right) \right| \leq (N^{1+C\e})^p.
\ee
We give the proof for \eqref{pqbh2}. The proof of \eqref{pqbh3} is the same except that it is slightly longer by one term of $\mathcal P_{\gamma,\al}$. Using \eqref{HHtt3}, with 
$$
f(I,J)=N^{-2p},\quad I=(i_1, \cdots, i_p), \quad J=(j_1, \cdots, j_p),
$$ 
and 
 $ \prod_{t=1}^p \mathcal P_{\gamma, {\bf k }}\, (S_{i_t j_t} \overline{S_{i_t j_t}})$ playing the role of $\prod_{t=1}^s G^{\gamma_{\rm max}} _{i_t j_t}$ in \eqref{HHtt3},  $\wt G$ being  $G^0$,  we have
\be \begin{split} \label{780k}
&\left|\E\sum_{I,J} N^{-2p}  \prod_{t=1}^p \mathcal P_{\gamma, {\bf k }}\, (S_{i_t j_t} \overline{S_{i_t j_t}})\right| \\
&\leq  \left|\E \sum_{I,J} N^{-2p} \prod_{t=1}^p \mathcal P_{\gamma, {\bf k }}\, (\wt G_{i_t j_t} \overline{\wt G_{i_t j_t}})\right|\\
& \quad +\max_{ {\bf k}, n, \gamma } (N^{-\phi/20})^{\sum_i |{\bf k_i}|} 
\left|\E \sum_{I,J} N^{-2p} \mathcal P_{\gamma_n, {\bf k}_n}\cdots \mathcal P_{\gamma_1, {\bf k}_1}  \left(\mathcal P_{\gamma, {\bf k }}\, (\wt G_{i_t j_t} \overline{\wt G_{i_t j_t}})\right)\right|+O(N^{-\zeta})
\end{split} \ee
where 
\be
{\bf k}_1\in \mathbb R^{2p}, \quad {\bf k}_2\in \mathbb R^{2p+|{\bf k}_1|}, \quad 
{\bf k}_3\in \mathbb R^{2p+|{\bf k}_1|+|{\bf k}_2|},\quad \cdots.
\ee
From \eqref{Gij estimate} and assumption on $z$ in \eqref{qglm}, with high probability, 
\be\label{781s}
|\wt G_{ij}|\leq N^{-1/3+2\e}+2\delta_{ij}.
\ee

Now, we estimate the term
\be\label{782s}
  \mathcal P_{\gamma_n, {\bf k}_n}\cdots \mathcal P_{\gamma_1, {\bf k}_1}  \left(\mathcal P_{\gamma, {\bf k }}\, (\wt G_{i_t j_t} \overline{\wt G_{i_t j_t}} )\right)
\ee
as in \eqref{dsb}. First, it is the sum of at most $C^{\sum |{\bf k}_i|+|\bf k|}=O(1)$ products of $\wt G_{ij}$ (possibly $i = j$), where in each product the number of $\wt G_{ij}$ is $\sum |{\bf k}_i|+2p { \,=\,} O(1)$. Since $\wt G$ satisfies a rough bound $|\wt G_{ij} (z)|\le \eta^{-1}\leq N$, we know \eqref{782s} is always bounded by $N^{O(1)}$. Since \eqref{781s} holds with high probability, when estimating \eqref{782s}, we may neglect the event that \eqref{781s} does not hold. For each product of above type and for any fixed $t$, the indices $i_t$ and  $j_t$ only appear twice each. Since ${\bf k}\neq 0$, they cannot attain the form $\wt G_{i_t, j_t} \overline{\wt G_{i_t, j_t}}$. Thus, they must appear as one of the following forms for some $a, b, c, d$, which comes from $\mathcal P$'s:
\be\label{Gabcd}
\wt G_{i_t a} \wt G_{b j_t} \overline{\wt G_{i_t j_t}}, \quad \wt G_{i_t j_t} \overline{\wt G_{i_t a}} \overline{\wt G_{b j_t}}, \quad G_{i_t a} G_{b j_t} \overline{\wt G_{i_t c}} \overline{\wt G_{d j_t}}.
\ee
For each case, after averaging over $1 \leq i_t, j_t \leq N$, i.e., applying $N^{-2} \sum_{i_t, j_t}$, these terms are bounded by $N^{-1+C\e}$. Thus, so far we have proved that, for each $t$, the term $G_{ij}$ with an index $i_t$ or $j_t$ contributes a factor $N^{-1+C\e}$ to 
\be\label{784d}
 \sum_{I,J} N^{-2p} \mathcal P_{\gamma_n, {\bf k}_n}\cdots \mathcal P_{\gamma_1, {\bf k}_1}  \left(\mathcal P_{\gamma, {\bf k }}\, (\wt G_{i_t j_t} \overline{\wt G_{i_t j_t}}) \right).
\ee
Similarly, the $G_{ij}$'s with indices $i_1, j_1, \cdots, i_p, j_p$ contribute a factor $(N^{-1+C\e})^p$ to \eqref{784d}. By \eqref{asXY}, it is bounded with $X+CN^{-1}$. For the other $G$'s  without indices $i_1, j_1, \cdots, i_p, j_p$, we simply bound them by a constant $C$. Therefore, we obtain that \eqref{784d} is bounded by $(N^{-1+C\e})^p$ with high probability. Then, the expectation of \eqref{784d} is less than $(N^{-1+C\e})^p$. Analogously, we can bound the first term in the right hand side of \eqref{780k} by $(N^{-1+C\e})^p$. Thus, we proved \eqref{pqbh2} and \eqref{pqbh3}, which implies \eqref{pqbh} with \eqref{boundcalA2}. We can complete the proof of \eqref{pqbh0} similarly.  

Now, we return to estimate $x^S-x^R$ in \eqref{yrsh}. First, we note that 
$$
\E |h_{ab}|^3\leq (\E|h_{ab}|^2\E|h_{ab}|^4)^{1/2}\le (\log N)N^{-3/2}.
$$
With \eqref{785}, we can see that there exists a constant $C$ such that
\begin{equation}
\E  |x^S - x^R|^3 \leq N^{  -5/2+C\e}
\end{equation}
for any sufficiently large $N$ independent of $\gamma$. Together with the fact that $\zeta_S$ is between $x^S$ and $x^R$, we get $|\zeta_S|\leq N^{C\e}$ (see \eqref{flld}) with high probability, hence
\begin{equation}
\left| \sum_{\gamma =1}^{\gamma_{\max}} \E \left[ F^{(3)}(\zeta_S) (x^S - x^R)^3 \right] \right| \leq N^{  - \phi+C\e},
\end{equation}
where we have used \eqref{jhw2} on $F$. We can estimate $\E \left[ F^{(3)}(\zeta_T) (x^T - x^R)^3 \right]$ analogously.

From \eqref{eq:telescopic_fn} and \eqref{eq:taylor}, it only remains to prove that, for $1\leq s\le 2$, 
\begin{equation}\label{788y}
\left| \E \left[ F^{(s)}(x_R) (x^S - x^R)^s\right] -\E \left[ F^{(s)}(x_R) (x^T - x^R)^s\right] \right| \leq N^{  -2- \phi+C\e}.
\end{equation}
Using \eqref{yrsh} again, recalling $(H_{\gamma-1})_{ab}$ has the distribution of $\wt H_{ab}$, we have 
\be\label{yrsh2}
 x^T  = x^R +
 \sum _{1\leq k\leq 3/\phi}\left( \sum_{{\bf k}\in    I^{2 }_{3/\phi, k}}
\eta^2\sum _{ij}\mathcal P_{\gamma, {\bf k}}\, (R_{i j}R_{ji})   \right)(-\wt h_{ab})^k+
O(  C N^{-3})
 \ee
with $\E h^k_{ab}=\E \wt h^k_{ab}$, $(1\leq k\leq 4)$. Thus, we obtain that
 \begin{align}
&\left|  \E \left[ F^{(s)}(x_R) (x^S - x^R)^s\right] -\E \left[ F^{(s)}(x_R) (x^T - x^R)^s\right] \right|
\\\nonumber
\leq & \left| \sum _{ k=5}^{9/\phi} \;
\sum_{\sum_{t=1}^s|{\bf k_t}|=k}
 \sum_{{\bf k}_t\in    I^{2s}_{3/\phi, k}}
\E \prod_{t=1}^s\left( 
\mathcal P_{\gamma, {\bf k}_t}\;x^R  \right)\right|  \left(\left|\E(-h_{ab})^k\right|+\left|\E(-\wt h_{ab})^k\right|\right)+
O(  C N^{-3}),
\end{align}
where in the right hand side we sum up $k$ from $5$. Since $\left|\E(-\wt h_{ab})^k\right| \leq (\log N)^CN^{-5/2}$ and $ \left|\E(-  h_{ab})^k\right| \leq (\log N)^CN^{-2-\phi}$ , using \eqref{785}, we obtain \eqref{788y} and complete the proof. 
\end{proof}

\section*{Acknowledgment} 
The authors would like to thank H.-T. Yau for helpful discussions.

\end{document}